\documentclass[thmsa,onecolumn,11pt,11pt]{article}
\usepackage{amssymb}
\usepackage{graphicx}
\usepackage{amsmath}

\newtheorem{theorem}{Theorem}

\newtheorem{corollary}[theorem]{Corollary}

\newtheorem{lemma}[theorem]{Lemma}

\newtheorem{proposition}[theorem]{Proposition}

\newenvironment{proof}[1][Proof]{\textbf{#1.} }{\ \rule{0.5em}{0.5em}}

\begin{document}

\title{Group Representations and High-Resolution Central
Limit Theorems for Subordinated Spherical Random Fields}
\date{July 10, 2007}

\author{Domenico MARINUCCI\thanks{%
Department of Mathematics, University of Rome \textquotedblleft
Tor Vergata\textquotedblright. E-mail:
\texttt{marinucc@mat.uniroma2.it}}\ \ and \
Giovanni PECCATI \thanks{%
Laboratoire de Statistique Th\'{e}orique et Appliqu\'{e}e,
Universit\'{e} Paris VI. E-mail:
\texttt{giovanni.peccati@gmail.com}}}
\date{July 10, 2007}
\maketitle

\begin{abstract}
\noindent We study the weak convergence (in the high-frequency
limit) of the frequency components associated with
Gaussian-subordinated, spherical and isotropic random fields. In
particular, we provide conditions for asymptotic Gaussianity and
we establish a new connection with random walks on the
hypergroup $\widehat{SO\left( 3\right) }$ (the dual of $SO\left( 3\right) $%
), which mirrors analogous results previously established for
fields defined on Abelian groups (see Marinucci and Peccati
(2007)). Our work is motivated by applications to cosmological
data analysis, and specifically by the probabilistic modelling and
the statistical analysis of the Cosmic Microwave Background
radiation, which is currently at the frontier of physical
research. To obtain our main results, we prove several fine
estimates involving convolutions of the so-called
\textsl{Clebsch-Gordan coefficients} (which are elements of\
unitary matrices connecting reducible representations of $SO\left(
3\right) $); this allows to intepret most of our asymptotic
conditions in terms of coupling of angular momenta in a quantum
mechanical system. Part of the proofs are based on recently
established criteria for the weak convergence of multiple
Wiener-It\^{o} integrals. 
\end{abstract}

\section{Introduction}

This paper deals with weak limit theorems involving the high-frequency
components (in the sense of the spherical harmonics decomposition) of random
fields defined on the unit sphere $\mathbb{S}^{2}$. Our results are
motivated by a number of mathematical issues arising in connection with the
probabilistic and statistical analysis of the Cosmic Microwave Background
radiation (see e.g. \cite{dodelson}). We start by giving a description of
our abstract mathematical framework, along with a sketch of the main results
of the paper; the subsequent Section \ref{ss : PhysMot} focuses on the
physical motivations and applications of our research. Here, and for the
rest of the paper, all random elements are defined on a suitable probability
space $\left( \Omega ,\mathcal{F},\mathbb{P}\right) $.

\subsection{General framework and outline of the main results}

We shall consider real-valued random fields $\{\widetilde{T}(x):x\in \mathbb{%
S}^{2}\}$ enjoying the following properties:%
\begin{equation}
\mathbb{E}\widetilde{T}(x)=0\text{ , }\mathbb{E}\widetilde{T}^{2}(x)<+\infty
\text{ \ and \ }\widetilde{T}(gx)\overset{law}{=}\widetilde{T}(x)\text{,}
\label{IntroDEF}
\end{equation}%
for all $x\in \mathbb{S}^{2}$ and all $g\in SO(3)$, where $\overset{law}{=}$
denotes equality in law (in the sense of stochastic processes). A field
verifying the last relation in (\ref{IntroDEF}) is usually called \textsl{%
isotropic }or \textsl{rotationally-invariant }(in law). It is a standard
result that the following spectral representation holds in the mean-square
sense:%
\begin{equation}
\widetilde{T}\left( x\right) =\sum_{l=0}^{\infty }\widetilde{T}%
_{l}(x)=\sum_{l=0}^{\infty }\sum_{m=-l}^{l}a_{lm}Y_{lm}\left( x\right) \text{%
,}  \label{specrap}
\end{equation}%
where $\left\{ Y_{lm}:l\geq 0\text{, }m=-l,...,l\right\} $ is the
collection of the spherical harmonics, and the $\left\{
a_{lm}\right\} $ are the associated (harmonic) Fourier coefficients. For $%
l\geq 0$, we also write $C_{l}\triangleq \mathbb{E}\left\vert
a_{lm}\right\vert ^{2}$, and we call the sequence $\left\{ C_{l}:l\geq
0\right\} $ the \textsl{angular power spectrum} of the random field $%
\widetilde{T}$ (note that $C_{l}$ does not depend on $m$ -- see e.g. \cite%
{BaMa}). For every $l\geq 0$, the field $\widetilde{T}_{l}$ provides the
projection of $\widetilde{T}$ on the subspace of $L^{2}(\mathbb{S}^{2},dx)$
spanned by the class $\left\{ Y_{lm}:m=-l,...,l\right\} $. The spherical
harmonics form an orthonormal basis of $L^{2}(\mathbb{S}^{2},dx)$ which can
be derived from the restriction to the sphere of harmonic polynomials. In
particular, in spherical coordinates $x=(\theta ,\varphi )$ they can be
written explicitly as: $Y_{00}\equiv 1/\sqrt{4\pi }$ and%
\begin{eqnarray}
Y_{lm}(\theta ,\varphi ) &=&\sqrt{\frac{2l+1}{4\pi }\frac{(l-m)!}{(l+m)!}}%
P_{lm}(\cos \theta )e^{im\varphi }\text{, }m\geq 0\text{ ,}  \label{IntroSH1}
\\
Y_{lm}(\theta ,\varphi ) &=&(-1)^{m}\overline{Y_{l,-m}}(\theta ,\varphi )%
\text{, }m<0,\text{ }0\leq \theta \leq \pi ,\text{ }0\leq \varphi <2\pi
\text{ ,}  \label{IntroSH2}
\end{eqnarray}%
where, for $l\geq 1$ and $m=0,1,2,...,l,$ $P_{lm}(\cdot )$ denotes the
Legendre polynomial of index $l,m,$ i.e.,%
\begin{equation}
P_{lm}(x)=(-1)^{m}(1-x^{2})^{m/2}\frac{d^{m}}{dx^{m}}P_{l}(x)\text{ , }%
P_{l}(x)=\frac{1}{2^{l}l!}\frac{d^{l}}{dx^{l}}(x^{2}-1)^{l}.
\label{Legendre}
\end{equation}%
For a discussion of these and other properties of the spherical harmonics
see e.g. \cite[Chapter 9]{Libo}, or \cite[Chapter 5]{VMK}. For $l\geq 0$,
the real-valued field $\widetilde{T}_{l}$ is called the $l$th \textsl{%
frequency component }of $\widetilde{T}$. The expansion (\ref{specrap}) can
be achieved by many different routes, for instance by a Karhunen-Lo\'{e}ve
argument or by means of the stochastic Peter-Weyl theorem, see for instance
\cite{Adler}, \cite{BaMaVa}, \cite{Leon} and \cite{PePy}. The random
harmonic coefficients $\left\{ a_{lm}\right\} $ appearing in (\ref{specrap})
form a triangular array of zero-mean random variables, which are
complex-valued for $m\neq 0$ and such that $\mathbb{E}a_{lm}\overline{%
a_{l^{\prime }m^{\prime }}}=\delta _{l}^{l^{\prime }}\delta
_{m}^{m^{\prime}}C_{l}$ (the bar denotes complex conjugation and
$\delta $ is Kronecker's symbol; note also that
$a_{lm}=(-1)^{m}\overline{a_{l-m}}$). For a Gaussian random field
$\widetilde{T}$ verifying (\ref{IntroDEF}), it is trivial that the
set $\left\{ a_{lm}\right\} $ is itself a complex-Gaussian array,
with independent elements for $m\geq 0$. It is a simple but
interesting fact that
the converse also holds, i.e. that, under an isotropy assumption on $%
\widetilde{T}$, the independence of the $a_{lm}$'s for $m\geq 0$ implies
Gaussianity, see \cite{BaMa}. Apart from this result, the behaviour of the
array $\left\{ a_{lm}\right\} $ and of the projections $\{\widetilde{T}%
_{l}\} $ for non-Gaussian isotropic fields is so far almost completely
unexplored and open for research, although such objects are highly relevant
for cosmological applications (see the next subsection). It should be
stressed that the coefficients $\left\{ a_{lm}\right\} $ depend on the
choice of coordinates and are not intrinsic to the field, although their law
is. In this sense, it is sometimes physically more sound to focus on the
behaviour of the sequence of projections $\{\widetilde{T}_{l}\}$, which are
indeed invariant with respect to the choice of coordinates.

In what follows, we focus on non-Gaussian fields $\widetilde{T}$ that are
\textsl{Gaussian-subordinated}, and we address the previous topic by
studying the asymptotic behaviour of $\left\{ a_{lm}\right\} $ and $\{%
\widetilde{T}_{l}\}$, as $l\rightarrow +\infty $. Recall that $\widetilde{T}$
is called \textsl{Gaussian-subordinated} whenever $\widetilde{T}\left(
x\right) =F\left( T\left( x\right) \right) $, where $F$ is a suitable
real-valued function, and $T$ is an isotropic spherical (real) Gaussian
field. In particular, our purpose is to establish sufficient (and sometimes,
also necessary) conditions on $F$ and on the law of $T$ to have that the
following two phenomena take place: (\textbf{I}) as $l\rightarrow +\infty $,
for a fixed $m$ and for an appropriate sequence $\tau _{1}\left( l\right) $ (%
$l\geq \left\vert m\right\vert $), the sequence $$\tau _{1}\left(
l\right) \times a_{lm} =\tau _{1}\left( l\right)
\int_{\mathbb{S}^{2}}F\left( T\left( z\right) \right)
\overline{Y_{lm}\left( z\right) }dz, \text{ \ \ } l\geq \left\vert
m\right\vert$$ converges in law to a Gaussian random variable
(real-valued for $m=0$, and complex-valued for $m\neq 0$);
(\textbf{II}) for a suitable real-valued sequence $\tau _{2}\left(
l\right) $ ($l\geq 0$) and for $l$ sufficiently large, the
finite-dimensional distributions of the field $$\tau _{2}\left(
l\right) \times \widetilde{T}_{l}\left( \cdot \right)=\tau
_{2}\left( l\right) \sum_{m=-l,...,l}a_{lm}Y_{lm}\left( \cdot
\right), $$
are close (for instance, in the sense of the Prokhorov
distance -- see \cite{Pesco}) to those of a real spherical
Gaussian field. Note that
both results (\textbf{I}) and (\textbf{II}) can be interpreted as CLTs%
\textsl{\ in the high-frequency }(or \textsl{high-resolution}) \textsl{sense}%
, since they involve Gaussian approximations and are established by letting
the frequency index $l$ diverge to infinity.

Our findings generalize previous results, obtained in \cite{MaPe}, for
fields defined on Abelian compact groups. One of our main tools is a result
concerning the Gaussian approximation of multiple Wiener-It\^{o} integrals
established in \cite{Pesco} (see also \cite{NuPe}, \cite{PeTaq2} and \cite%
{PT}). These CLTs can be seen as a simplification of the combinatorial
\textsl{method of diagrams and cumulants }(see e.g. \cite{Surg}). These
techniques, combined with the use of group representation theory, lead to
one of the main contributions of this paper: the derivation of sufficient
(or necessary and sufficient) conditions for (\textbf{I}) and (\textbf{II}),
expressed in terms of convolutions of \textsl{Clebsch-Gordan coefficients}
(see e.g. \cite[Ch. 4]{VMK}), which are the elements of unitary matrices
connecting specific reducible representations of $SO\left( 3\right) $.
Clebsch-Gordan coefficients are widely used in quantum mechanics, and admit
a well-known interpretation in terms of probability amplitudes related to
the coupling of angular momenta in a quantum mechanical system (see \cite%
{Libo}, \cite{VMK} or Sections \ref{S : Clebsch-Gordan} and \ref{S :
Roynette} below). We will also show that many of our conditions can be
alternatively restated in terms of `bridges' of random walks on $\widehat{%
SO\left( 3\right) }$ (the dual of $SO\left( 3\right) $). The definition of
such random walks differs from the classic one given in \cite{GKR}, although
the two approaches can be related through the notion of \textsl{mixed
quantum state} (see Section \ref{S : Roynette}). Note that an analogous
connection with random walks on $\mathbb{Z}^{d}$ was pointed out in \cite%
{MaPe}.

\subsection{Cosmological motivations\label{ss : PhysMot}}

The Cosmic Microwave Background radiation (hereafter CMB) can be viewed as a
relic radiation of the Big Bang, providing maps of the primordial Universe
before the formation of any of the current structures (approximately, $%
3\times 10^{5}$ years after the Big Bang); as such, it is acknowledged as a
goldmine of information for fundamental physics. Many satellite experiments
involving hundred of physicists throughout the world are devoted to the
construction of spherical maps of the CMB radiation, and for pioneering work
in this area G. Smoot and J. Mather were awarded the Nobel Prize for Physics
in 2006 -- see for instance \texttt{http://map.gsfc.nasa.gov/} for more
details.

The crucial point is that most cosmological models imply that the CMB
radiation is the realization of a random field $\{\widetilde{T}(x):x\in
\mathbb{S}^{2}\}$, verifying the three conditions in (\ref{IntroDEF}); each $%
x\in \mathbb{S}^{2}$ corresponds to a direction in which the CMB radiation\
is measured. The isotropic property can be seen as a consequence of
Einstein's \textsl{cosmological principle}, roughly stating that, on
sufficiently large distance scales, the Universe looks identical everywhere
in space (homogeneity) and appears the same in every direction (isotropy). A
central issue in modern cosmology relates therefore to the distribution of
the CMB random field $\widetilde{T}$, which is predicted to be (close to)
Gaussian by some models for the dynamics at primordial epochs (for instance,
by the so-called \textsl{inflationary scenario}), and non-Gaussian by other
models, where fluctuations are generated by topological defects arising in
phase transitions of a thermodynamical nature -- see for instance \cite%
{dodelson}. Many testing procedures have been proposed to tackle this issue;
in some form, they all rely asymptotically on the behaviour of the field at
the highest frequencies (see for instance \cite{Bart}, \cite{Marinucci} and
the references therein). This is a sort of unescapable, foundational issue
in Cosmology. By definition, the latter is a science based on a single
realization, e.g. our Universe or the trace of its primordial structure in
the form of the CMB\ radiation, which is observed at higher and higher
resolutions. As such, an asymptotic theory for statistical tests is possible
only in the sense of observations at higher and higher frequencies (smaller
and smaller scales) becoming available as the experiments become more
sophisticated. In particular, any satellite experiment measuring the CMB
radiation can reconstruct the spherical harmonic developement appearing in (%
\ref{specrap}) only up to a finite frequency $l_{\max }$, the quantity $\pi
/l_{\max }$ representing approximately the \textsl{angular resolution} of
the experiment (the pioneering satellite COBE (1993) could reach a frequency
$l_{\max }\simeq 20$, WMAP (2003, 2006) improved this limit to $l_{\max
}\simeq 600/800$, and Planck (to be launched in 2008) is expected to reach $%
l_{\max }\simeq 2500/3000$). In order for such procedures to yield
consistent outcomes, one should therefore figure out what is the limiting
behaviour of $\{\widetilde{T}_{l}\}$, for $l>>0$, under different
distributional assumptions on $\widetilde{T}$. Some Monte Carlo evidence
(see for instance \cite{MaPi} and the references therein) has suggested that
this behaviour may be close to Gaussian even in circumstances where the
underlying field $\widetilde{T}$ clearly is not. The investigation of this
issue is necessary for rigorous inference on CMB data, and in particular for
non-Gaussianity tests. The relevance of the asymptotic behaviour of the $\{%
\widetilde{T}_{l}\}$, however, goes much beyond the issue of such tests, and
relates indeed to the whole statistical analysis of CMB -- which is largely
dominated by likelihood approaches (see \cite{efst}).

We stress by now that the results we provide cover models that are quite
relevant for cosmological applications, for instance the so called \textsl{%
Sachs-Wolfe model}, which represents the standard starting model for the
inflationary scenario (see for instance \cite{Bart}, \cite{dodelson}). In
its simplest version, this model implies that the CMB is a straightforward
quadratic transformation of an underlying Gaussian field, i.e.%
\begin{equation}
\widetilde{T}(x)=T(x)+f_{NL}\left\{ T(x)^{2}-\mathbb{E}T(x)^{2}\right\}
\text{ , \ \ }x\in \mathbb{S}^{2}\text{,}  \label{swmodel}
\end{equation}%
where $f_{NL}$ is a nonlinearity parameter depending on constants from
particle physics and $T$ is Gaussian and isotropic. As a special case, our
results do allow for a complete characterization of the high-frequency
behaviour of models such as (\ref{swmodel}), and in this sense they are
immediately applicable in the cosmological literature.

\subsection{Plan}

In Section 2 we provide some background material on isotropic
random fields on the sphere. Section 3 is devoted to a discussion
on representation theory for the group of rotations $SO(3)$ and
the so-called Clebsch-Gordan coefficients, which will play a
crucial role in the analysis to follow. In Section 4 we state and
prove a general CLT result for the spherical harmonics
coefficients and the high-frequency components of a field arising
from polynomial transformations of arbitrary order of a
subordinating Gaussian process. In Section 5 we provide a more
detailed analysis of necessary and sufficient condition for the
CLT to hold in the case of quadratic and cubic transformations; we
also highlight the connections between our conditions and the
theory of random walks on hypergroups. The interplay with random
walks on hypergroups is further explored in Section 6, where some
comparisons with the existing literature are provided, and some
physical interpretations of our conditions in terms of randomly
interacting quantum particles are given. In Section 7, we turn our
attention to more explicit conditions on the angular power
spectrum, and we discuss an exponential/algebraic duality which
parallels to some extent some earlier findings in the Abelian
case.

\section{Preliminaries on Gaussian and Gaussian-subordinated\
i\-so\-tro\-pic fields\label{S : GaussSub}}

As in the Introduction, we denote by $\mathbb{S}^{2}$ the unit sphere $%
\mathbb{S}^{2}=\left\{ x\in \mathbb{R}^{3}:\left\Vert x\right\Vert
=1\right\} $. For every rotation $g\in SO\left( 3\right) $ and every $x\in
\mathbb{S}^{2}$, the symbol $gx$ indicates the canonical action of $g$ on $x$
(see \cite[Ch. 1]{VMK}, as well as Section \ref{S : Clebsch-Gordan} below,
for further details). We will systematically write $dx$ for the Lebesgue
measure on $\mathbb{S}^{2}$, and we denote by $L^{2}\left( \mathbb{S}%
^{2},dx\right) $ the class of complex-valued functions on $\mathbb{S}^{2}$
which are square-integrable with respect to $dx$. We denote by $\left\{
Y_{lm}:l\geq 0\text{, }m=-l,...,l\right\} $ the basis of $L^{2}\left(
\mathbb{S}^{2},dx\right) $ given by spherical harmonics, as defined via (\ref%
{IntroSH1}) and (\ref{IntroSH2}). From now on, we shall denote by $T=\left\{
T\left( x\right) :x\in \mathbb{S}^{2}\right\} $ a centered, real-valued and
\textsl{Gaussian} random field parametrized by $\mathbb{S}^{2}$. We also
suppose that $T$ is \textsl{isotropic}, that is, for every $g\in SO\left(
3\right) $ one has that $T\left( x\right) \overset{law}{=}T\left( gx\right) $%
, where the equality holds in the sense of finite dimensional distributions.
To simplify the notation, we also assume that $\mathbb{E}T\left( x\right)
^{2}=1$. Following e.g. \cite{BaMa} (but see also \cite{BaMaVa}, \cite{PePy}
and \cite{Pyc}), one deduces from isotropy that $T$ admits the spectral
decomposition
\begin{equation}
T\left( x\right) =\sum_{l=0}^{\infty }\sum_{m=-l}^{l}a_{lm;1}Y_{lm}\left(
x\right) =\sum_{l=0}^{\infty }T_{l}\left( x\right) \text{, \ \ }x\in \mathbb{%
S}^{2}\text{,}  \label{spectral-G}
\end{equation}%
where $a_{lm;1}\triangleq \int_{\mathbb{S}^{2}}T\left( x\right) \overline{%
Y_{lm}\left( x\right) }dx$ (the role of the subscript \textquotedblleft $%
lm;1 $\textquotedblright\ will be clarified in the following discussion), $%
T_{l}\left( x\right) $ $\triangleq $ $\sum_{m=-l}^{l}a_{lm;1}Y_{lm}\left(
x\right) $, and the convergence takes place in $L^{2}\left( \mathbb{P}%
\right) $ for every fixed $x$, as well as in $L^{2}\left( \mathbb{P}\otimes
dx\right) $. The next result gives a simple and very useful characterization
of the joint law of the complex-valued array $\left\{ a_{lm;1}:l\geq 0\text{%
, }m=-l,...,l\right\} $. For every $z\in \mathbb{C}$, the symbols $\Re
\left( z\right) $ and $\Im \left( z\right) $ indicate, respectively, the
real and the imaginary part of $z$.

\begin{proposition}
\label{P : BaMa}Let $T$ be the centered, isotropic and Gaussian random field
appearing in (\ref{spectral-G}). Then: (i) for every $l\geq 0$ the random
variable $a_{l0;1}$ is real-valued, centered and Gaussian; (ii) for every $%
l\geq 1$, and every $m=1,...,l$, the random variable $a_{lm;1}$ is
complex-valued and such that $a_{lm;1}$ $=$ $\left( -1\right) ^{m}\overline{%
a_{l-m;1}}$, and moreover $\mathbb{E(}\Re \left( a_{lm;1}\right) ^{2})$ $=$ $%
\mathbb{E(}\Im \left( a_{lm;1}\right) ^{2})$ $=$ $\mathbb{E(}a_{l0;1}^{2})/2$
$=$ $C_{l}/2$, for some constant $C_{l}\in \left[ 0,+\infty \right) $ not
depending on $m$, and%
\begin{equation}
\mathbb{E(}\Re \left( a_{lm;1}\right) \Im \left( a_{lm;1}\right) )=0;
\label{aa'}
\end{equation}%
(iii) for every $l\geq 1$ and every $m=-l,...,l$, the random coefficient $%
a_{lm;1}$ is independent of $a_{l^{\prime }m^{\prime };1}$ for every $%
l^{\prime }\geq 0$ such that $l^{\prime }\neq l$ and every $m^{\prime
}=-l^{\prime },...,l^{\prime }$. By noting $C_{0}\triangleq \mathbb{E(}%
a_{00;1}^{2})$, one also has the relation
\begin{equation}
1=\mathbb{E}\left[ T\left( x\right) ^{2}\right] =\sum_{l=0}^{\infty }\frac{%
2l+1}{4\pi }C_{l}\text{.}  \label{var}
\end{equation}
\end{proposition}

The reader is referred to \cite{BaMa} for a proof of Proposition \ref{P :
BaMa}, as well as for several converse statements. Here, we shall only
stress that formula (\ref{var}) is a consequence of the well-known relation
(see e.g. \cite{VMK})
\begin{equation}
\sum_{m=-l}^{l}Y_{lm}(x)\overline{Y_{lm}(y)}=\frac{2l+1}{4\pi }P_{l}(\cos
\left\langle x,y\right\rangle )\text{, \ \ }x,y\in \mathbb{S}^{2}\text{,}
\label{angle}
\end{equation}%
where $\left\langle x,y\right\rangle $ is the angle between $x$ and $y$.\
Observe that property (\ref{aa'}) implies that $\Re \left( a_{lm;1}\right) $
and $\Im \left( a_{lm;1}\right) $ are independent centered Gaussian random
variables. Moreover, the combination of (\ref{aa'}) and point (iii) in the
statement of Proposition \ref{P : BaMa} yields that $\mathbb{E(}a_{lm;1}%
\overline{a_{l^{\prime }m^{\prime };1}})$ $=$ $0$, $\forall \left(
l,m\right) $ $\neq $ $\left( l^{\prime },m^{\prime }\right) $. Finally, it
is also evident that points (i)-(iii) in the previous statement imply that
the law of an isotropic Gaussian field such as $T$ is completely
characterized by its angular power spectrum $\left\{ C_{l}:l\geq 0\right\} $%
. To avoid trivialities, we will always work under the following assumption:

\smallskip \textbf{Assumption. }The angular power spectrum $\left\{
C_{l}:l\geq 0\right\} $ is such that $C_{l}>0$ for every $l$. \smallskip

Note that the results of this paper could be extended without difficulties
(but at the cost of an heavier notation) to the case of a power spectrum
such that $C_{l}\neq 0$ for infinitely many $l$'s. In the subsequent
sections, we shall obtain high-frequency CLTs for centered isotropic
spherical fields that are \textsl{subordinated} to the Gaussian field $T$
defined above.

\smallskip

\textbf{Definition A }(\textit{Subordinated fields}).\textbf{\ }Let $%
L_{0}^{2}(\mathbb{R}$, $e^{-z^{2}/2}dz)$ indicate the class of
real-valued functions $F\left( z\right) $ on $\mathbb{R}$, which
are square-integrable with respect to the measure $e^{-z^{2}/2}dz$
and such that $\int F\left( z\right) e^{-z^{2}/2}dz=0$. A
(centered) random field $\widetilde{T}= \{\widetilde{T}\left(
x\right) :x\in \mathbb{S}^{2}\} $ is said to be
\textsl{subordinated} to the Gaussian field $T$ appearing in
(\ref{specrap})
if there exists $F\in L_{0}^{2}(\mathbb{R}$, $e^{-z^{2}/2}dz)$ such that $%
\widetilde{T}\left( x\right) $ $=F\left[ T\right] \left( x\right) $,\ $%
\forall x\in \mathbb{S}^{2}$, where the symbol $F\left[ T\right] \left(
x\right) $ stands for $F\left( T\left( x\right) \right) $. Whenever $%
\widetilde{T}$ is subordinated, we will rather use the notation $F\left[ T%
\right] \left( x\right) $ instead of $\widetilde{T}\left( x\right) $, in
order to emphasize the role of the function $F$. Of course, if $F\left(
z\right) =z$, then $F\left[ T\right] \left( x\right) $ $=\widetilde{T}\left(
x\right) $ $=T\left( x\right) $.

\smallskip

It is immediate to check that, since $T$ is isotropic, a subordinated field $%
F\left[ T\right] \left( \cdot \right) $ as in Definition A is necessarily
isotropic. As a consequence, following again \cite{BaMa} or \cite{PePy}, one
deduces that $F\left[ T\right] $ admits the spectral representation
\begin{equation}
F\left[ T\right] \left( x\right) =\sum_{l=0}^{\infty
}\sum_{m=-l}^{l}a_{lm}\left( F\right) Y_{lm}\left( x\right)
=\sum_{l=0}^{\infty }F\left[ T\right] _{l}\left( x\right) \text{, \ \ }x\in
\mathbb{S}^{2}\text{,}  \label{SpecSub}
\end{equation}%
with convergence in $L^{2}\left( \mathbb{P}\right) $ (for fixed $x$) and in $%
L^{2}\left( \Omega \times \mathbb{S}^{2}\text{, }\mathbb{P}\otimes dx\right)
$. Here,
\begin{eqnarray}
a_{lm}\left( F\right) &\triangleq &\int_{\mathbb{S}^{2}}F\left[ T\right]
\left( y\right) \overline{Y_{lm}\left( y\right) }dy\text{,\ and}
\label{subCoeff} \\
F\left[ T\right] _{l}\left( x\right) &\triangleq
&\sum_{m=-l}^{l}a_{lm}\left( F\right) Y_{lm}\left( x\right) .
\label{subProj}
\end{eqnarray}%
The complex-valued array $\left\{ a_{lm}\left( F\right) :l\geq 0,\text{ \ }%
m=-l,...,l\right\} $ always enjoys the following properties (\textbf{a})-(%
\textbf{c}): (\textbf{a}) for every $l\geq 0$, the random variable $%
a_{l0}\left( F\right) $ is real-valued, centered and Gaussian; (\textbf{b})
for every $l\geq 1$, and every $m=1,...,l$, the random variable $%
a_{lm}\left( F\right) $ is complex-valued and such that
\begin{align*}
a_{lm}\left( F\right) & =\left( -1\right) ^{m}\overline{a_{l-m}\left(
F\right) }\text{ \ ; \ }\mathbb{E(}\Re \left( a_{lm}\left( F\right) \right)
\Im \left( a_{lm}\left( F\right) \right) )=0 \\
\mathbb{E(}\Re \left( a_{lm}\left( F\right) \right) ^{2})& =\mathbb{E(}\Im
\left( a_{lm}\left( F\right) \right) ^{2})=\mathbb{E(}a_{l0}\left( F\right)
^{2})/2=C_{l}\left( F\right) /2,
\end{align*}%
where the finite constant $C_{l}\left( F\right) \geq 0$ depends solely on $F$
and $l$; (\textbf{c}) $\mathbb{E(}a_{lm}\left( F\right) $ $\times $ $%
\overline{a_{l^{\prime }m^{\prime }}\left( F\right) })$ $=0$, $\forall
\left( l,m\right) $ $\neq $ $\left( l^{\prime },m^{\prime }\right) $. Note
that, in general, it is no longer true that $\Re \left( a_{lm}\left(
F\right) \right) $ and $\Im \left( a_{lm}\left( F\right) \right) $ are
independent random variables. Moreover, we state the following consequence
of \cite[Th. 7]{BaMa}: \textit{for every }$l\geq 1$, \textit{the coefficients%
} $\left( a_{l0}\left( F\right) ,...,a_{ll}\left( F\right) \right) $ \textit{%
are stochastically independent if, and only if, they are Gaussian. }Also, $%
\mathbb{E(}F\left[ T\right] \left( x\right) ^{2})=\sum_{l=0}^{\infty }\frac{%
2l+1}{4\pi }C_{l}\left( F\right) $.

In the subsequent sections, a crucial role will be played by the class of
\textit{Hermite polynomials}. Recall (see e.g. \cite[p. 20]{Janss}) that the
sequence $\left\{ H_{q}:q\geq 0\right\} $ of Hermite polynomials is defined
by the differential relation
\begin{equation}
H_{q}\left( z\right) =\left( -1\right) ^{q}e^{\frac{z^{2}}{2}}\frac{d^{q}}{%
dz^{q}}e^{-\frac{z^{2}}{2}}\text{, \ \ }z\in \mathbb{R}\text{, \ }q\geq 0;
\label{Her}
\end{equation}%
it is well-known that the sequence $\{ \left( q!\right)
^{-1/2}H_{q}:q\geq 0\} $ defines an orthonormal basis of the space $%
L^{2}(\mathbb{R},\left( 2\pi \right) ^{-1/2}e^{-z^{2}/2}dz)$. When a
subordinated field has the form (for $q\geq 2$) $H_{q}\left[ T\right] \left(
x\right) $, $x\in \mathbb{S}^{2}$ (that is, when $F=H_{q}$ in Definition A),
we will use the shorthand notation:
\begin{eqnarray}
T^{\left( q\right) }\left( x\right) &\triangleq &H_{q}\left[ T\right] \left(
x\right) ,\text{ \ }x\in \mathbb{S}^{2}\text{,}  \label{Short1} \\
a_{lm;q} &\triangleq &a_{lm}\left( H_{q}\right) \text{,}  \label{Short1.5} \\
T_{l}^{\left( q\right) }\left( x\right) &\triangleq &H_{q}\left[ T\right]
_{l}\left( x\right) ,\text{ \ }l\geq 1\text{, }x\in \mathbb{S}^{2}\text{,}
\label{Short2} \\
\overline{T}_{l}^{\left( q\right) }\left( x\right) &\triangleq &Var\left(
T_{l}^{\left( q\right) }\left( x\right) \right) ^{-1/2}T_{l}^{\left(
q\right) }\left( x\right) ,\text{ \ }l\geq 1\text{, }x\in \mathbb{S}^{2}%
\text{,}  \label{Short2.5} \\
\widetilde{C}_{l}^{(q)} &\triangleq &C_{l}\left( H_{q}\right) =\mathbb{E}%
|a_{lm;q}|^{2}\text{, }l\geq 1\text{, }m=-l,...,l\text{.}  \label{Short3}
\end{eqnarray}

To justify our notation (\ref{Short1})--(\ref{Short3}), we recall that for
every fixed $x$ the random variable $H_{q}\left[ T\right] \left( x\right)
=H_{q}\left( T\left( x\right) \right) $ is just the $q$th \textsl{Wick power
}of $T\left( x\right) $ (see for instance \cite{Janss}). We conclude the
section with an easy Lemma, that will be used in Section \ref{S: CLT}.

\begin{lemma}
\label{L : Pl}Let $F\left[ T\right] \left( x\right) $, $x\in \mathbb{S}^{2}$%
, be an (isotropic) subordinated field as in Definition A. Then, for every $%
l\geq 1$ one has the following:

\begin{enumerate}
\item The random field $x\mapsto F\left[ T\right] _{l}\left( x\right) $
defined in (\ref{subProj}) is real-valued and isotropic;

\item For every fixed $x\in \mathbb{S}^{2}$, $F\left[ T\right] _{l}\left(
x\right) $ $\overset{law}{=}$ $\sqrt{\frac{2l+1}{4\pi }}a_{l0}\left(
F\right) $, where the coefficient $a_{l0}\left( F\right) $ is defined
according to (\ref{subCoeff}), and consequently $\mathbb{E(}F\left[ T\right]
_{l}\left( x\right) ^{2})$ $=$ $\frac{2l+1}{4\pi }C_{l}\left( F\right) $;

\item The normalized random field
\begin{equation}
\overline{F\left[ T\right] }_{l}\left( x\right) =\left[ \frac{\left(
2l+1\right) C_{l}\left( F\right) }{4\pi }\right] ^{-1/2}F\left[ T\right]
_{l}\left( x\right)  \label{SubNorm}
\end{equation}%
has a covariance structure given by: for every $x,y\in \mathbb{S}^{2}$,%
\begin{equation}
\mathbb{E}\left( \overline{F\left[ T\right] }_{l}\left( x\right) \times
\overline{F\left[ T\right] }_{l}\left( x\right) \right) =P_{l}\left( \cos
\left\langle x,y\right\rangle \right) \text{,}  \label{covSubNorm}
\end{equation}%
where $P_{l}\left( \cdot \right) $ is the $l$\textrm{th} Legendre polynomial
defined in (\ref{Legendre}) and, as before, $\left\langle x,y\right\rangle $
is the angle between $x$ and $y$.
\end{enumerate}
\end{lemma}

\begin{proof}
Point 1. is straightforward. To prove point 2. define (in polar coordinates)
$x_{0}=\left( 0,0\right) $ and use the isotropy property stated at point 1.
to write
\begin{equation*}
F\left[ T\right] _{l}\left( x\right) \overset{law}{=}F\left[ T\right]
_{l}\left( x_{0}\right) =\sum_{m=-l}^{l}a_{lm}\left( F\right) Y_{lm}\left(
x_{0}\right) =\sqrt{\frac{2l+1}{4\pi }}a_{l0}\left( F\right) \text{,}
\end{equation*}%
since (\ref{IntroSH1}) implies that $Y_{lm}\left( x_{0}\right) =\sqrt{\left(
2l+1\right) /4\pi }\delta _{m}^{0}$. Finally, to prove relation (\ref%
{covSubNorm}) we use (\ref{angle}) to deduce that, for every $x,y\in \mathbb{%
S}^{2}$,%
\begin{equation*}
\mathbb{E}(F\left[ T\right] _{l}\left( x\right) F\left[ T\right]
_{l}\left( y\right) ) =C_{l}\left( F\right) \frac{2l+1}{4\pi
}P_{l}(\cos \left\langle x,y\right\rangle )\text{,}
\end{equation*}%
thus giving the desired conclusion (recall that $P_{l}\left( 1\right) =1$).
\end{proof}

For instance, a first consequence of Lemma \ref{L : Pl} is that, for every $%
q\geq 2$,
\begin{equation}
\mathbb{E(}T_{l}^{\left( q\right) }\left( x\right) ^{2})=\left( 2l+1\right)
\widetilde{C}_{l}^{\left( q\right) }/4\pi  \label{tq}
\end{equation}%
where we used the notation introduced at (\ref{Short1})-(\ref{Short3}), so
that $\overline{T}_{l}^{\left( q\right) }\left( x\right) $ $=$ $[\left(
2l+1\right) \widetilde{C}_{l}^{(q)}/4\pi ]^{-1/2}$ $T_{l}^{\left( q\right)
}\left( x\right)$.

The main aim of the subsequent sections is to provide an accurate solution
to the following problems (\textbf{P-I})--(\textbf{P-III}).

\smallskip

\noindent%
(\textbf{P-I}) For a fixed $q\geq 2$, find conditions on the power spectrum $%
\left\{ C_{l}:l\geq 0\right\} $ of $T$, to have that the subordinated
process $T^{\left( q\right) }=\left\{ T^{\left( q\right) }\left( x\right)
:x\in \mathbb{S}^{2}\right\} $ defined in (\ref{Short1}) is such that, for
every $x\in \mathbb{S}^{2}$,%
\begin{equation}
\sqrt{\left( 2l+1\right) \widetilde{C}_{l}^{\left( q\right) }/4\pi }\times
T_{l}^{\left( q\right) }\left( x\right) \underset{l\rightarrow +\infty }{%
\overset{law}{\longrightarrow }}N\text{,}  \label{CLT1}
\end{equation}%
where $N\ $is a centered standard Gaussian random variable.

\noindent%
(\textbf{P-II}) Under the conditions found at (\textbf{P-I}), study the
asymptotic behaviour, as $l\rightarrow +\infty $, of the vector
\begin{equation}
\sqrt{\left( 2l+1\right) \widetilde{C}_{l}^{\left( q\right) }/4\pi }\times
\left( T_{l}^{\left( q\right) }\left( x_{1}\right) ,...,T_{l}^{\left(
q\right) }\left( x_{k}\right) \right) \text{,}  \label{vector}
\end{equation}%
for every $x_{1},...,x_{k}\in \mathbb{S}^{2}$.

\noindent%
(\textbf{P-III}) Combine (\textbf{P-I}) and (\textbf{P-II}) to study the
asymptotic behaviour (in particular, the asymptotic Gaussianity), as $%
l\rightarrow +\infty $, of vectors of the type
\begin{equation}
\sqrt{\left( 2l+1\right) C_{l}\left( F\right) /4\pi }\times \left( F\left[ T%
\right] _{l}\left( x_{1}\right) ,...,F\left[ T\right] _{l}\left(
x_{k}\right) \right) \text{,}  \label{vector2}
\end{equation}%
for every $x_{1},...,x_{k}\in \mathbb{S}^{2}$ and every $F\in
L_{0}^{2}\left( \mathbb{R},e^{-z^{2}/2}dz\right) $.

\smallskip

Note that Problems (\textbf{P-I})-(\textbf{P-III}) are stated in increasing
order of generality. We observe also the following fact: since (\ref%
{covSubNorm}) holds, and since the limit of $P_{l}\left( \left\langle
x,y\right\rangle \right) $ ($l\rightarrow +\infty $) does not exist in
general, it will not be possible to prove that the vectors in (\ref{vector})
and (\ref{vector2}) converge in law to some Gaussian limit. However, by
using the results developed in \cite{Pesco}, we will be able to establish
conditions under which the laws of such vectors are \textquotedblleft
asymptotically close\textquotedblright\ to a sequence of $k$-dimensional
Gaussian distributions. As already mentioned, to study (\textbf{P-I})--(%
\textbf{P-III}) we shall use estimates involving the so-called \textsl{%
Clebsch-Gordan }coefficients, that are elements of unitary matrices
connecting some reducible representations of $SO\left( 3\right) $. The
definition and the analysis of some crucial properties of Clebsch-Gordan
coefficients are the object of the next section.

\section{A primer on Clebsch-Gordan coefficients\label{S : Clebsch-Gordan}}

In this subsection, we need to review some basic representation theory
results for $SO(3)$, the group of rotations in $\mathbb{R}^{3}$. We refer
the reader to standard textbooks (for instance, \cite{VMK} and \cite{VilKly}%
) for further details, as well as for any unexplained notion or definition.
It should be stressed that most of our arguments below could be extended to
general compact groups with known representations; however, throughout the
following we shall stick to the group of rotations $SO(3)$, mainly for the
sake of notational simplicity.

We recall first that each element $g\in SO(3)$ can be parametrized by the
set $(\alpha ,\beta ,\gamma )$ of so-called \textsl{Euler angles}, where $%
0\leq \alpha <2\pi ,$ $0\leq \beta \leq \pi $ and $0\leq \gamma <2\pi $. In
these coordinates, a complete set of irreducible matrix representations for $%
SO(3)$ is provided by the so-called \textsl{Wigner's }$D$\textsl{\ matrices}
$D^{l}(\alpha ,\beta ,\gamma )$, of dimensions ($2l+1)\times (2l+1)$ for $%
l=0,1,2,...$ -- see \cite[Ch. 4]{VMK} for an analytic expression.
Here, we simply point out that the elements of $D^{l}(\alpha
,\beta ,\gamma )$ are related to the spherical harmonics by the
relationship
\begin{equation}
D_{m0}^{l}(\alpha ,\beta ,\gamma )=(-1)^{m}\sqrt{\frac{4\pi }{2l+1}}%
Y_{l-m}(\beta ,\alpha )=\sqrt{\frac{4\pi }{2l+1}}Y_{lm}^{\ast }(\beta
,\alpha )\text{ ,}  \label{spherwig}
\end{equation}%
from which it is not difficult to show how the usual spectral representation
for random fields on the spheres (for instance (\ref{specrap}) and (\ref%
{spectral-G})) is really just the stochastic Peter-Weyl Theorem on $\mathbb{S%
}^{2}=SO(3)/SO(2)$. The reader is referred e.g. to \cite{VilKly} and \cite%
{Varadara} for further discussions on the Peter-Weyl Theorem, and to \cite%
{BaMa}, \cite{BaMaVa} and \cite{PePy} for several related probabilistic
results.

It follows from standard representation theory that we can exploit
the family $\{ D^{l}\} _{l=0,1,,2,...}$ to build alternative
(reducible) representations, either by taking the tensor product family $%
\{ D^{l_{1}}\otimes D^{l_{2}}\} _{l_{1},l_{2}}$, or by considering
direct sums $\{ \oplus _{l=|l_{2}-l_{1}|}^{l_{2}+l_{1}}D^{l}\}
_{l_{1},l_{2}}$; these
representations have dimensions $(2l_{1}+1)(2l_{2}+1)$ $\times $ $%
(2l_{1}+1)(2l_{2}+1)$ and are unitarily equivalent, whence there exists a
unitary matrix $C_{l_{1}l_{2}}$ such that%
\begin{equation}
\left\{ D^{l_{1}}\otimes D^{l_{2}}\right\} =C_{l_{1}l_{2}}\left\{ \oplus
_{l=|l_{2}-l_{1}|}^{l_{2}+l_{1}}D^{l}\right\} C_{l_{1}l_{2}}^{\ast }\text{.}
\label{clebun}
\end{equation}%
Here, $C_{l_{1}l_{2}}$ is a $\left\{ (2l_{1}+1)(2l_{2}+1)\times
(2l_{1}+1)(2l_{2}+1)\right\} $ block matrix with blocks $%
C_{l_{1}(m_{1})l_{2}}^{l}$ of dimensions $(2l_{2}+1)\times (2l+1)$, $%
m_{1}=-l_{1},...,l_{1}$. The elements of such a block are indexed by $m_{2}$
(over rows) and $m$ (over columns). More precisely,%
\begin{eqnarray*}
C_{l_{1}l_{2}} &=&\left[ C_{l_{1}(m_{1})l_{2}\cdot }^{l\cdot }\right]
_{m_{1}=-l_{1},...,l_{1};l=|l_{2}-l_{1}|,...,l_{2}+l_{1}} \\
C_{l_{1}(m_{1})l_{2}.}^{l.} &=&\left\{ C_{l_{1}m_{1}l_{2}m_{2}}^{lm}\right\}
_{m_{2}=-l_{2},...,l_{2};m=-l,...,l}\text{ .}
\end{eqnarray*}

The \textsl{Clebsch-Gordan coefficients} for $SO(3)$ are then defined as $%
\{\!C_{l_{1}m_{1}l_{2}m_{2}}^{lm}\!\} $, that is, as the elements
of the unitary matrices $C_{l_{1}l_{2}}$ (note that such matrices
are real-valued, and so are the $C_{l_{1}m_{1}l_{2}m_{2}}^{lm}$).
These coefficients were introduced in Mathematics in the XIX
century, as motivated by the analysis of invariants in Algebraic
Geometry; in the 20th century, they have gained
an enormous importance in the quantum theory of angular momentum, where $%
C_{l_{1}m_{1}l_{2}m_{2}}^{lm}$ represents the \textsl{probability amplitude}
that two quantum particles with total angular momentum $l_{1}$ and $l_{2}$
and momentum projections on the $z$-axis $m_{1}$ and $m_{2}$ are coupled to
form a system with total angular momentum $l$ and projection $m$ (see e.g.
\cite{Libo}). Their use in the analysis of isotropic random fields is much
more recent, see for instance \cite{Hu} and the references therein. Explicit
expressions for the Clebsch-Gordan coefficients of $SO(3)$ are known, but
they are in general hardly manageable (see e.g. \cite[Section 8.2]{VMK}).
However, these expressions become somewhat neater when $m_{1}=m_{2}=m_{3}=0$%
, in which case one has the relations: $C_{l_{1}0l_{2}0}^{l_{3}0}=0$, when $%
l_{1}+$ $l_{2}+$ $l_{3}$ is odd, and, for $l_{1}+l_{2}+l_{3}$ even,%
\begin{eqnarray*}
C_{l_{1}0l_{2}0}^{l_{3}0} &=&\frac{(-1)^{\frac{l_{1}+l_{2}-l_{3}}{2}}\sqrt{%
2l_{3}+1}\left[ (l_{1}+l_{2}+l_{3})/2\right] !}{\left[ (l_{1}+l_{2}-l_{3})/2%
\right] !\left[ (l_{1}-l_{2}+l_{3})/2\right] !\left[ (-l_{1}+l_{2}+l_{3})/2%
\right] !} \\
&&\times \left\{ \frac{%
(l_{1}+l_{2}-l_{3})!(l_{1}-l_{2}+l_{3})!(-l_{1}+l_{2}+l_{3})!}{%
(l_{1}+l_{2}+l_{3}+1)!}\right\} ^{1/2}.
\end{eqnarray*}

The Clebsch-Gordan coefficients enjoy also a nice set of symmetry and
orthogonality properties which will play a crucial role in our results to
follow (see \cite{Marinucci} and \cite{MarPTRF} for an account of such properties).
Note in particular that the Clebsch-Gordan coefficients are
different from zero only if $m_{1}+m_{2}=m$ and $|l_{2}-l_{1}|\leq l\leq
l_{1}+l_{2}$ (the \emph{triangle conditions})$.$ Also, from unitary
equivalence we deduce that%
\begin{equation}
\sum_{m_{1},m_{2}}\!
C_{l_{1}m_{1}l_{2}m_{2}}^{lm}\!C_{l_{1}m_{1}l_{2}m_{2}}^{l^{\prime
}m^{\prime }} \!=\! \delta _{l}^{l^{\prime} }\!\delta _{m}^{m^{\prime} }%
\text{ and }\sum_{l,m}\!C_{l_{1}m_{1}l_{2}m_{2}}^{lm}\!C_{l_{1}m_{1}^{\prime
}l_{2}m_{2}^{\prime }}^{lm}\!=\!\delta _{m_{1}}^{m_{1}^{\prime }}\delta
_{m_{2}}^{m_{2}^{\prime}}\!.  \label{orto1}
\end{equation}

\textbf{Remark on Notation. }Depending on the notational convenience, we
write sometimes sums of Clebsch-Gordan coefficients without specifying the
range of the indices $l$ and/or $m$. In such cases, the range of the sums is
conventionally taken to be the set of indices where the Clebsch-Gordan
coefficients are different from zero. For instance, in (\ref{orto1}) one
should read: $\sum_{m_{1},m_{2}}=\sum_{m_{1}=-l_{1},...,l_{1}}%
\sum_{m_{2}=-l_{2},...,l_{2}}$ and $\sum_{l,m}=\sum_{l=0}^{+\infty
}\sum_{m=-l,...,l}$. Similar conventions are adopted (without further
notice) throughout the paper. We recall also that the Clebsch-Gordan
coefficients are equivalent, up to a normalization factor, to the \emph{%
Wigner's} \emph{3j coefficients, }which are used\emph{\ }in
related works such as \cite{MarPTRF}.

The Clebsch-Gordan coefficients play a crucial role in the evaluation of
integrals involving products of spherical harmonics. In particular, the
so-called \textsl{Gaunt integral} gives%
\begin{equation}
\int_{\mathbb{S}^{2}}\!Y_{l_{1}m_{1}}\left( x\right)\! Y_{l_{2}m_{2}}\left(
x\right)\! \overline{Y_{lm}\left( x\right) }dx \!=\!\sqrt{\frac{\left(
2l_{1}\!+\!1\right)\! \left( 2l_{2}\!+\!1\right) }{4\pi \left( 2l+1\right) }}%
C_{l_{1}m_{1}l_{2}m_{2}}^{lm}C_{l_{1}0l_{2}0}^{l0}\text{ }.  \label{gauint}
\end{equation}%
Relation (\ref{gauint}) can be established using (\ref{spherwig}), (\ref%
{clebun}) and resorting to standard orthonormality properties of
the elements of group representations -- see \cite[Expression
5.9.1.4]{VMK}. More generally, define
\begin{equation}
\mathcal{G}\left\{ l_{1},m_{1};...;l_{r},m_{r}\right\} \triangleq \int_{%
\mathbb{S}^{2}}Y_{l_{1}m_{1}}\left( x\right) \cdot \cdot \cdot
Y_{l_{r}m_{r}}\left( x\right) dx\text{,}  \label{supergaunt}
\end{equation}%
and call the quantity $\mathcal{G}\left\{
l_{1},m_{1};...;l_{r},m_{r}\right\} $ a \textsl{generalized Gaunt integral}$%
. $ Then, iterating the previous argument, for $q\geq 3$ it can be shown
that (by using for instance \cite[Expression 5.6.2.12]{VMK})%
\begin{eqnarray}
&&\quad \quad \mathcal{G}\left\{ l_{1},m_{1};...;l_{q},m_{q};l,-m\right\}
\label{megagaunt} \\
&=&\sum_{L_{1}...L_{q-2}}\sum_{M_{1}...M_{q-2}}\left\{
\prod_{i=1}^{q-3}\left( \sqrt{\frac{2l_{i+2}+1}{4\pi }}%
C_{L_{i}0l_{i+2}0}^{L_{i+1}0}C_{L_{i}M_{i}l_{i+2}m_{i+2}}^{L_{i+1}M_{i+1}}%
\right) \right\}  \notag \\
&&\times \sqrt{\frac{(2l_{1}+1)(2l_{2}+1)}{4\pi (2l+1)}}%
C_{l_{1}0l_{2}0}^{L_{1}0}C_{l_{1}m_{1}l_{2}m_{2}}^{L_{1}M_{1}}
\!\sqrt{\frac{2l_{q}+1}{4\pi }}
C_{L_{q-2}0l_{q}0}^{l0}C_{L_{q-2}M_{q-2}l_{q}m_{q}}^{lm} \text{,}
\notag
\end{eqnarray}%
where, for $q=3$, we have used the convention
$\Pi_{i=1}^{0}\equiv0$. Note that expressions such as
(\ref{megagaunt}) imply that the generalized Gaunt integrals of
the type (\ref{supergaunt}) are indeed real-valued. To
simplify the expression (\ref{megagaunt}), let us introduce the coefficients%
\begin{equation*}
C_{l_{1},m_{1};...;l_{p}m_{p}}^{\lambda _{1},\lambda _{2},...,\lambda
_{p-1};\mu }\!\triangleq \!\sum_{\mu _{1}=-\lambda _{1}}^{\lambda
_{1}}\!...\!\sum_{\mu _{p-2}=-\lambda _{p-2}}^{\lambda
_{p-2}}C_{l_{1},m_{1},l_{2},m_{2}}^{\lambda _{1},\mu _{1}}C_{\lambda
_{1},\mu _{1};l_{3},m_{3}}^{\lambda _{2},\mu _{2}}\cdot \cdot \cdot
C_{\lambda _{p-2},\mu _{p-2};l_{p},m_{p}}^{\lambda _{p-1},\mu }.
\end{equation*}

These coefficients are themselves the elements of unitary matrices
connecting tensor product and direct sum representations of $SO(3)$, and
thus it follows easily that the following orthonormality conditions hold
\begin{equation}
\sum_{m_{1},...m_{p}}\left\{ C_{l_{1},m_{1};...;l_{p}m_{p}}^{\lambda
_{1},\lambda _{2},...,\lambda _{p-1};\mu }\right\} ^{2}=\sum_{\lambda
_{1}}...\sum_{\lambda _{p-1}}\sum_{\mu =-\lambda _{p-1}}^{\lambda
_{p-1}}\left\{ C_{l_{1},m_{1};...;l_{p}m_{p}}^{\lambda _{1},\lambda
_{2},...,\lambda _{p-1};\mu }\right\} ^{2}=1\text{ ;}  \label{ortoorto}
\end{equation}%
it is important to note that due to the conditions $m_{1}+m_{2}=m_{3}$ the
sums may actually vanish, for instance%
\begin{equation}
C_{l_{1},0;...;l_{p}0}^{\lambda _{1},\lambda _{2},...,\lambda
_{p-1};0}=C_{l_{1},0,l_{2},0}^{\lambda _{1},0}C_{\lambda
_{1},0;l_{3},0}^{\lambda _{2},0}\cdot \cdot \cdot C_{\lambda
_{p-2},0;l_{p},0}^{\lambda _{p-1},0}\text{ .}  \label{0_conv}
\end{equation}

We have also that
\begin{eqnarray}
&&\mathcal{G}\left\{ l_{1},m_{1};...;l_{q},m_{q};l,-m\right\}
\label{gagaunt} \\
&=&\sqrt{\frac{4\pi }{2l+1}}\left\{ \prod_{i=1}^{q}\sqrt{\frac{2l_{i}+1}{%
4\pi }}\right\}
\sum_{L_{1}...L_{q-2}}C_{l_{1},0;...;l_{q}0}^{L_{1},L_{2},...,L_{q-2},l;0}C_{l_{1},m_{1};...;l_{q}m_{q}}^{L_{1},L_{2},...,L_{q-2},l;m}.
\notag
\end{eqnarray}

\smallskip

\textbf{Remark. }The coefficients $C_{l_{1},m_{1};...;l_{p}m_{p}}^{\lambda
_{1},\lambda _{2},...,\lambda _{p-1};\mu }$ defined above admit a physical
interpretation in terms of coupling of angular momenta in a quantum
mechanical system. Consider indeed a system composed of $p$ particles, say $%
\alpha _{1},...,\alpha _{p}$, such that $\alpha _{i}$ has total angular
momentum equal to $l_{i}$, and projection on the $z$-axis given by $m_{i}$.
Then, the coefficient $C_{l_{1},m_{1};...;l_{p}m_{p}}^{\lambda _{1},\lambda
_{2},...,\lambda _{p-1};\mu }$ is exactly the probability amplitude of the
intersection of the following $p-1$ events $\mathbf{E}_{1}$,..., $\mathbf{E}%
_{p-1}$:

\noindent $\mathbf{E}_{1}$ $=$ $\{\alpha _{1}$ and $\alpha _{2}$ couple to
form a particle $\eta _{1}$ with total angular momentum $\lambda _{1}\}$, $%
\mathbf{E}_{2}$ $=$ $\{\eta _{1}$ couples with $\alpha _{3}$ to form a
particle $\eta _{2}$ with total angular momentum $\lambda _{2}\}$,...,
\textbf{E}$_{i}=$ $\{\eta _{i-1}$ couples with $\alpha _{i+1}$ to form a
particle $\eta _{i}$ with total angular momentum $\lambda _{2}\}$,...,
\textbf{E}$_{p-1}=$ $\{\eta _{p-2}$ couples with $\alpha _{p}$ to form a
particle with total angular momentum $\lambda _{p-1}$ and projection $\mu $
on the $z$-axis$\}$.

\smallskip

In the sequel, we shall also need the so-called \textsl{Wigner }$6j$\textsl{%
\ (or Racah) coefficients}, which are related to the Clebsch-Gordan by the
identity (see (\cite[Eq. 9.1.1.8]{VMK}))%
\begin{equation}
\left\{
\begin{array}{ccc}
\!l_{1}\! & l_{2}\! & l_{3}\! \\
\!l_{4}\! & l_{5}\! & l_{6}\!%
\end{array}%
\right\} \!=\!K\!\left(l_{1},...,\!l_{6}\!\right)\! \sum_{\substack{ %
m_{1}m_{3}  \\ m_{4}m_{6}}}\!C_{l_{1}m_{1}l_{2}m_{2}}^{l_{3}m_{3}}%
\!C_{l_{1}m_{1}l_{6}m_{6}}^{l_{5}m_{5}}
\!C_{l_{3}m_{3}l_{4}m_{4}}^{l_{5}m_{5}}%
\!C_{l_{2}m_{2}l_{4}m_{4}}^{l_{6}m_{6}}\!.  \label{wig6j}
\end{equation}%
where $K\left( l_{1},...,l_{6}\right) =\left[ (2l_{3}+1)(2l_{6}+1)\right]
^{-1/2}(-1)^{l_{1}+l_{2}+l_{4}+l_{5}}$ (note that the previous sum does not
involve $m_{2}$ and $m_{5}$, because of the general relation: $C_{\alpha
t_{1}\beta t_{2}}^{\gamma t_{3}}=0$, whenever $t_{3}\neq t_{1}+t_{2}$).
Although the Wigner's 6j coefficients play themselves a very important role
in Quantum Mechanics and Representation Theory, for brevity's sake we avoid
a full discussion on their properties; the interested reader can consult (%
\cite[Ch.9]{VMK}) or (\cite[pp. 529-542]{VilKly}).

\section{High-frequency CLTs: conditions in terms of Gaunt integrals\label%
{S: CLT}}

The aim of this section is to obtain conditions for high-frequency CLTs in
terms of Gaunt integrals of the type (\ref{megagaunt}). We start by
focussing on Hermite polynomials, and then we deal with general subordinated
fields.

\subsection{Hermite subordination\label{SS : Hermite}}

We focus on the spherical field $T^{\left( q\right) }$ ($q\geq 2$) defined
in (\ref{Short1}), which is obtained by composing the Gaussian field $T$ in (%
\ref{specrap}) with the $q$th Hermite polynomial $H_{q}$ (or, equivalently,
by taking the $q$th Wick power of the random variable $T\left( x\right) $
for every $x$). Our first purpose is to characterize the asymptotic
Gaussianity (when $l\rightarrow +\infty $) of the spherical harmonic
coefficients $\left\{ a_{lm;q}\right\} $ defined in (\ref{Short1.5}).

\begin{theorem}
\label{teo1}Fix $q\geq 2$.

\textrm{1.} For every $l\geq 1$, the positive constant $\widetilde{C}%
_{l}^{\left( q\right) }$ in (\ref{Short3}) (which does not depend on $m$)
equals the quantity%
\begin{eqnarray}
&&q!\!\sum_{l_{1},m_{1}}\!\cdot \!\cdot \!\cdot
\!\sum_{l_{q},m_{q}}C_{l_{1}}C_{l_{2}}\cdot \cdot \cdot C_{l_{q}}\left\vert
\mathcal{G}\left\{ l_{1},m_{1};...;l_{q},m_{q};l,-m\right\} \right\vert ^{2}
\label{VAR} \\
&&= \!q!\!\sum_{l_{1},...,l_{q}=0}^{\infty }C_{l_{1}}\!\cdot
\!\cdot \!\cdot
\!C_{l_{q}}\frac{4\pi }{2l+1}\left\{ \prod_{i=1}^{q}\frac{2l_{i}+1}{4\pi }%
\right\} \sum_{L_{1}...L_{q-2}}\left\{
C_{l_{1},0;...;l_{q}0}^{L_{1},L_{2},...,L_{q-2},l;0}\right\} ^{2}
\label{VAR2}
\end{eqnarray}%
for every $m=-l,...,l$, where the (generalized) Gaunt integral $\mathcal{G}%
\left\{ \cdot \right\} $ is defined via (\ref{supergaunt}).

\textrm{2.} Fix $m\neq 0$. As $l\rightarrow +\infty $, the following two
conditions (\textbf{A}) and (\textbf{B}) are equivalent: (\textbf{A})%
\begin{equation}
(\widetilde{C}_{l}^{\left( q\right) })^{-1/2}\times a_{lm;q}\overset{law}{%
\rightarrow }N+iN^{\prime }\text{,}  \label{as0}
\end{equation}%
where $N,N^{\prime }\sim \mathcal{N}\left( 0,1/2\right) $ are independent; (%
\textbf{B}) for every $p=\frac{q-1}{2}+1,...,q-1$, if $q-1$ is even, and
every $p=q/2,...,q-1$ if $q-1$ is odd%
\begin{eqnarray}
&&(\widetilde{C}_{l}^{\left( q\right) })^{-2}\sum_{n_{1},j_{1}}\cdot \cdot
\cdot \sum_{n_{2\left( q-p\right) },j_{2\left( q-p\right) }}C_{j_{1}}\cdot
\cdot \cdot C_{j_{2\left( q-p\right) }}\left\vert \sum_{l_{1},m_{1}}\cdot
\cdot \cdot \sum_{l_{p},m_{p}}C_{l_{1}}\cdot \cdot \cdot C_{l_{p}}\right.
\notag \\
&&\text{ \ \ }\times \!\mathcal{G}\!\left\{
l_{1},m_{1};...;l_{p},m_{p};j_{1},n_{1};...;j_{q-p},n_{q-p};l,-m\right\}
\times  \label{as11} \\
&&\text{ \ \ }\left. \times \!\mathcal{G}\!\left\{
l_{1},m_{1};...;l_{p},m_{p};j_{q-p+1},n_{q-p+1};...;j_{2\left( q-p\right)
},n_{2\left( q-p\right) };l,-m\right\} ^{^{^{^{^{^{{}}}}}}}\!\right\vert
^{2} \! \rightarrow 0 \!  \notag
\end{eqnarray}

\textrm{3.} Let $N$ be a centered Gaussian random variable with unitary
variance. As $l\rightarrow +\infty $, the CLT
\begin{equation}
(\widetilde{C}_{l}^{\left( q\right) })^{-1/2}\times a_{l0;q}\overset{law}{%
\rightarrow }N  \label{as00}
\end{equation}%
takes place if, and only if, the asymptotic condition (\ref{as11}) holds for
$m=0$ and for every $p=\frac{q-1}{2}+1,...,q-1$, if $q-1$ is even, and every
$p=q/2,...,q-1$ if $q-1$ is odd.
\end{theorem}

\begin{proof}
Consider a standard Brownian motion $W=\left\{ W_{t}:t\in \left[ 0,1\right]
\right\} $, and denote by $L_{\mathbb{C}}^{2}\left( \left[ 0,1\right]
\right) $ $=$ $L_{\mathbb{C}}^{2}\left( \left[ 0,1\right] ,d\lambda \right) $
the class of complex-valued and square integrable functions on $\left[ 0,1%
\right] $, with respect to the Lebesgue measure $d\lambda $. Now select a
complex-valued family $\left\{ g_{lm}:l\geq 0\text{, \ }-l\leq m\leq
l\right\} \subseteq L_{\mathbb{C}}^{2}\left( \left[ 0,1\right] \right) $
with the following five properties: (1) $g_{l0}$ is real for every $l\geq 0$%
, (2) $g_{lm}=\left( -1\right) ^{m}\overline{g_{l-m}}$, (3)
$\int g_{lm}\overline{g_{l^{\prime }m^{\prime }}}d\lambda $ $=$ $0,$\ $%
\forall \left( l,m\right) $ $\neq $ $\left( l^{\prime },m^{\prime }\right) $%
, (4) $\int \Re \left( g_{lm}\right) \Im \left( g_{lm}\right)
d\lambda $ $=$ $0$, (5) $\int \Re \left( g_{lm}\right)
^{2}d\lambda $ $=$ $\int \Im \left( g_{lm}\right) ^{2}d\lambda $
$=$ $\int g_{l0}^{2}d\lambda /2$ $=C_{l}/2$, where $\left\{
C_{l}:l\geq 0\right\} $ is the power spectrum of the Gaussian
field $T$. According to Proposition \ref{P : BaMa}, the following
identity
in law holds:%
\begin{equation*}
\left\{ a_{lm;1}:l\geq 0\text{, \ }-l\leq m\leq l\right\} \overset{law}{=}%
\left\{ I_{1}\left( g_{lm}\right) :l\geq 0\text{, \ }-l\leq m\leq l\right\} ,
\end{equation*}%
where $I_{1}\left( g_{lm}\right)
=\int_{0}^{1}g_{lm}dW=\int_{0}^{1}\Re (g_{lm})dW$ $+$
$\mathrm{i}\int_{0}^{1}\Im (g_{lm})dW$ is the usual
(complex-valued) Wiener-It\^{o} integral of $g_{lm}$ with respect
to $W$. From this last relation, it also follows that, in the
sense of stochastic processes, $T\left( x\right) $
$\overset{law}{=}$ $I_{1}\left( \sum_{l=0}^{\infty
}\sum_{m=-l}^{l}g_{lm}Y_{lm}\left( x\right) \right) $ (note that
the function $z\mapsto \sum_{l,m}$$g_{lm}(z)Y_{lm}\left( x\right)
$ is real-valued for every fixed $x\in \mathbb{S}^{2}$ and with
norm equal to $1$). Now define $L_{s,\mathbb{C}}^{2}\left( \left[
0,1\right]
^{q}\right) $ to be the class of complex-valued and symmetric functions on $%
\left[ 0,1\right] ^{q}$, that are square-integrable with respect to Lebesgue
measure. For every $f\in L_{s,\mathbb{C}}^{2}\left( \left[ 0,1\right]
^{q}\right) $, we define $I_{q}\left( f\right) =I_{q}\left( \Re (f)\right) +%
\mathrm{i}I_{q}\left( \Im (f)\right) $ to be the multiple Wiener-It\^{o}
integral, of order $q$, of $f$ with respect to the Brownian motion $W$ (see
e.g. \cite[Ch. 1]{NualartBook}, or \cite{Janss}). From the previous
discussion it follows that, for every $q\geq 2$,
\begin{equation}
T^{\left( q\right) }\left( x\right) =H_{q}\left( T\left( x\right) \right)
\overset{law}{=}I_{q}\left[ \left\{ \sum_{l=0}^{\infty
}\sum_{m=-l}^{l}g_{lm}Y_{lm}\left( x\right) \right\} ^{\otimes q}\right] ,
\label{gr}
\end{equation}%
where the equality in law holds in the sense of finite dimensional
distributions and, for every $f\in L_{\mathbb{C}}^{2}\left( \left[ 0,1\right]
\right) $, we use the notation $f^{\otimes q}\left( a_{1},...,a_{q}\right) $
$=$ $f\left( a_{1}\right) $ $\times $ $\cdot \cdot \cdot $ $\times $ $%
f\left( a_{q}\right) .$ Note that, to obtain the last equality in (\ref{gr}%
), we used the well-known relation (see e.g. \cite{Janss}): for every
real-valued $f\in L_{\mathbb{R}}^{2}\left( \left[ 0,1\right] \right) $ such
that $\left\Vert f\right\Vert _{L_{\mathbb{R}}^{2}\left( \left[ 0,1\right]
\right) }=1$, it holds that $H_{q}\left[ I_{1}\left( f\right) \right] $ $=$ $%
I_{q}\left( f^{\otimes q}\right) $. Now set $h_{l,m}^{\left( q\right) }$ $=$
$\left( -1\right) ^{m}\sum_{l_{1},m_{1}}$ $\cdot \cdot \cdot $ $%
\sum_{l_{q},m_{q}}g_{l_{1}m_{1}}$ $\cdot \cdot \cdot $ $g_{l_{q}m_{q}}$ $%
\mathcal{G\{}l_{1},m_{1};...;$ $l_{q},m_{q};l,-m\}$, so that%
\begin{equation}
a_{lm;q}\overset{law}{=}\int_{S^{2}}I_{q}\left[ \left\{ \sum_{l=0}^{\infty
}\sum_{m=-l}^{l}g_{lm}Y_{lm}\left( x\right) \right\} ^{\otimes q}\right]
\overline{Y_{lm}\left( x\right) }dx=I_{q}\left[ h_{l,m}^{\left( q\right) }%
\right]   \label{fubz}
\end{equation}%
so that (\ref{VAR}) follows immediately from the well-known isometry
relation:%
\begin{equation*}
\mathbb{E}\left[ \left\vert I_{q}\left[ h_{l,m}^{\left( q\right) }\right]
\right\vert ^{2}\right] =q!\left\Vert h_{l,m}^{\left( q\right) }\right\Vert
_{L^{2}\left( \left[ 0,1\right] ^{q}\right) }^{2}
\end{equation*}%
(to obtain (\ref{fubz}) we interchanged stochastic and
deterministic integration, by means of a standard stochastic Fubini
argument). To prove that (\ref{VAR2}) is equal to (\ref{VAR}), observe first
that (\ref{ortoorto}) yields that
\begin{equation*}
\sum_{m_{1}=-l_{1}}^{l_{1}}\cdot \cdot \cdot
\sum_{m_{q}=-l_{q}}^{l_{q}}C_{l_{1},m_{1};...;l_{q}m_{q}}^{L_{1},L_{2},...,L_{q-2},l;m}C_{l_{1},m_{1};...;l_{q}m_{q}}^{L_{1}^{\prime },L_{2}^{\prime },...,L_{q-2}^{\prime },l;m}=\delta _{L_{1}}^{L_{1}^{\prime }}...\delta _{L_{q-2}}^{L_{q-2}^{\prime }}
\end{equation*}%
(the RHS of the previous expression does not depend on $m$). Then, use (\ref%
{gagaunt}) to deduce that
\begin{eqnarray}
&&\sum_{m_{1}=-l_{1}}^{l_{1}}\cdot \cdot \cdot \sum_{m_{q}=-l_{q}}^{l_{q}}%
\mathcal{G}\left\{ l_{1},m_{1};...;l_{q},m_{q};l,-m\right\} ^{2}  \notag \\
&=&\frac{4\pi }{2l+1}\left\{ \prod_{i=1}^{q}\frac{2l_{i}+1}{4\pi }\right\}
\sum_{L_{1}...L_{q-2}}\left\{
C_{l_{1},0;...;l_{q}0}^{L_{1},L_{2},...,L_{q-2},l;0}\right\} ^{2}.  \notag
\end{eqnarray}%
This proves Point 1 in the statement. To prove Point 2, recall that,
according to \cite[Proposition 6]{MaPe}, relation (\ref{as0}) holds if, and
only if,
\begin{equation*}
(\widetilde{C}_{l}^{\left( q\right) })^{-2}\left\Vert h_{l,m}^{\left(
q\right) }\otimes _{p}\overline{h_{l,m}^{\left( q\right) }}\right\Vert
_{L^{2}(\left[ 0,1\right] ^{2\left( q-p\right) })}^{2}\rightarrow 0,
\end{equation*}%
for every $p=1,...,q-1$, where the complex-valued (and not necessarily
symmetric) function $h_{l,m}^{\left( q\right) }\otimes _{p}\overline{%
h_{l,m}^{\left( q\right) }}$ (which is an element of $L^{2}(\left[ 0,1\right]
^{2\left( q-p\right) })$) is defined as the \textsl{contraction}%
\begin{eqnarray}
&&h_{l,m}^{\left( q\right) }\otimes _{p}\overline{h_{l,m}^{\left( q\right) }}%
\left( a_{1},...,a_{2\left( q-p\right) }\right)   \label{contpr} \\
&=&\int_{\left[ 0,1\right] ^{p}}h_{l,m}^{\left( q\right) }\left( \mathbf{x}%
_{p},a_{1},...,a_{q-p}\right) \overline{h_{l,m}^{\left( q\right) }\left(
\mathbf{x}_{p},a_{q-p+1},...,a_{2\left( q-p\right) }\right) }d\mathbf{x}_{p},
\notag
\end{eqnarray}%
for every $( a_{1},...,a_{2\left( q-p\right) }) \in \left[ 0,1%
\right] ^{2\left( q-p\right) }$, where $d\mathbf{x}_{p}$ is the
Lebesgue measure on $\left[ 0,1\right] ^{p}$. Since, trivially,
$\Vert h_{l,m}^{\left( q\right) }\otimes _{p}\overline{h_{l,m}^{\left( q\right) }}%
\Vert ^{2}$ $=$ $\Vert h_{l,m}^{\left( q\right) }\otimes _{q-p}%
\overline{h_{l,m}^{\left( q\right) }} \Vert ^{2}$ (we stress that,
in the last equality, the first norm is taken in $L^{2}( \left[
0,1\right] ^{2\left( q-p\right)})$, whereas the second is in
$L^{2}( \left[ 0,1\right]^{2p})$ ), one deduces that it is
sufficient to check that
the norm of $h_{l,m}^{\left( q\right) }\otimes _{p}%
\overline{h_{l,m}^{\left( q\right) }}$ is asymptotically
negligeable for every $p=\frac{q-1}{2}+1,...,q-1$, if $q-1$ is
even, and every $p=q/2,...,q-1$ if $q-1$ is odd. It follows that
the result is proved once it is shown that, for every $p$ in such
range, the norm $\Vert h_{l,m}^{\left( q\right) }\otimes _{p}\overline{%
h_{l,m}^{\left( q\right) }}\Vert ^{2}$ equals the multiple sum
appearing in (\ref{as11}). To see this, use (\ref{contpr}) to
deduce that
(recall that Gaunt integrals are real-valued)%
\begin{eqnarray*}
&&h_{l,m}^{\left( q\right) }\otimes _{p}\overline{h_{l,m}^{\left( q\right) }}%
\left( a_{1},...,a_{2\left( q-p\right) }\right)  \\
&=&\sum_{n_{1},j_{1}}\cdot \cdot \cdot \sum_{n_{2\left( q-p\right)
},j_{2\left( q-p\right) }}g_{j_{1}n_{1}}\cdot \cdot \cdot g_{j_{q-p}n_{q-p}}%
\overline{g_{j_{q-p+1}n_{q-p+1}}\cdot \cdot \cdot g_{j_{2\left( q-p\right)
}n_{2\left( q-p\right) }}} \\
&&\sum_{l_{1},m_{1}}\cdot \cdot \cdot \sum_{l_{p},m_{p}}C_{l_{1}}\cdot \cdot
\cdot C_{l_{p}}\mathcal{G}\left\{
l_{1},m_{1};...;l_{p},m_{p};j_{1},n_{1};...;j_{q-p},n_{q-p};l,-m\right\}  \\
&&\text{ \ }\mathcal{G}\left\{
l_{1},m_{1};...;l_{p},m_{p};j_{q-p+1},n_{q-p+1};...;j_{2\left( q-p\right)
},n_{2\left( q-p\right) };l,-m\right\} ,
\end{eqnarray*}%
and the result is obtained by using the orthogonality properties of the $%
g_{jn}$'s. Point 3 in the statement is proved in exactly the same way, by
first observing that $a_{l0;q}$ is a real-valued random variable, and then
by applying Theorem 1 in \cite{NuPe}.
\end{proof}

\textbf{Remark}.\textbf{\ }One has the relation $\mathbb{E}\left[ T^{\left(
q\right) }\left( x\right) ^{2}\right] =q!\left[ E\left\{ T(x)^{2}\right\} %
\right] ^{q}$. This equality can be proved in two ways: (i) by exploiting
the representation of $T^{\left( q\right) }\left( x\right) $ as a multiple
Wiener-It\^{o} integral, or (ii) by using the equality $\mathbb{E}\left[
T^{\left( q\right) }\left( x\right) ^{2}\right] =\sum_{l}\frac{2l+1}{4\pi }%
\widetilde{C}_{l}^{(q)}$, and the by expanding
$\widetilde{C}_{l}^{(q)}$ according to Theorem \ref{teo1}, so that
one can apply the orthogonality relations (\ref{ortoorto}).

\smallskip

Now recall that, according to part 2 of Lemma \ref{L : Pl}, $T_{l}^{\left(
q\right) }\left( x\right) \overset{law}{=}\sqrt{\frac{2l+1}{4\pi }}a_{l0;q}$%
, so that relation (\ref{tq}) holds. This gives immediately a first
(exhaustive) solution to Problem (\textbf{P-I}), as stated in Section \ref{S
: GaussSub}.

\begin{corollary}
\label{C : PunctualCLT}For every $q\geq 2$ the following conditions are
equivalent:

\begin{enumerate}
\item The CLT (\ref{CLT1}) holds for every $x\in \mathbb{S}^{2}$;

\item The asymptotic relation (\ref{as11}) takes place for $m=0$ and for
every $p=\frac{q-1}{2}+1,...,q-1$, if $q-1$ is even, and every $%
p=q/2,...,q-1 $ if $q-1$ is odd.
\end{enumerate}
\end{corollary}

To deal with Problem (\textbf{P-II}) of Section \ref{S : GaussSub}, we
recall the notation $\overline{T}_{l}^{\left( q\right) }$ (indicating the $l$%
th normalized frequency component of $T^{\left( q\right) }$) introduced in (%
\ref{Short2.5}). We also introduce (for every $l\geq 1$) the \textsl{%
normalized }$l$\textsl{th frequency component} of the Gaussian field $T$,
which is defined as
\begin{equation}
\overline{T}_{l}\left( x\right) =\frac{T_{l}\left( x\right) }{%
Var(T_{l}\left( x\right) )^{1/2}}=\frac{T_{l}\left( x\right) }{(\frac{2l+1}{%
4\pi }C_{l})^{1/2}}\text{, \ \ }x\in \mathbb{S}^{2}\text{.}  \label{ti_barra}
\end{equation}

According to Lemma \ref{L : Pl} (in the special case $F\left(
z\right) =z$), $\overline{T}_{l}$ is a real-valued, isotropic,
centered and Gaussian field. Moreover, one has that $\mathbb{E}[
\overline{T}_{l}\left( x\right)
\overline{T}_{l}\left( y\right) ] $ $=$ $\mathbb{E}[ \overline{T}%
_{l}^{\left( q\right) }\left( x\right) \overline{T}_{l}^{\left( q\right)
}\left( y\right) ] =P_{l}\left( \left\langle x,y\right\rangle \right) $%
, for every $q\geq 2$ and every $l\geq 1$. The next result -- which gives an
exhaustive solution to Problem (\textbf{P-II}) -- states that, whenever
Condition 1 (or, equivalently, Condition 2) in the statement of Corollary %
\ref{C : PunctualCLT} is verified (and without\textsl{\ any }additional
assumption), the \textquotedblleft distance\textquotedblright\ between the
finite dimensional distributions of the normalized field $\overline{T}%
_{l}^{\left( q\right) }$ and those of $\overline{T}_{l}$ converge
to zero. For every $k\geq 1$, we denote by $\mathbf{P}
(\mathbb{R}^{k}) $ the class of all probability measures on
$\mathbb{R}^{k}$. We say that a metric $\gamma \left( \cdot ,\cdot
\right) $ \textsl{metrizes the weak convergence} \textsl{on}
$\mathbf{P}( \mathbb{R}^{k}) $ whenever the following double
implication holds for every $Q\in \mathbf{P}( \mathbb{R}^{k}) $
and every $\left\{ Q_{l}:l\geq 1\right\} \subset
\mathbf{P}(\mathbb{R}^{k}) $ (as $l\rightarrow +\infty $): $%
\gamma \left( Q_{l},Q\right) \rightarrow 0$ if, and only if,
$Q_{l}$ converges weakly to $Q$. The quantity $\gamma(P,Q)$ is
sometimes called the $\gamma$\textsl{--distance} between $P$ and
$Q$.

\begin{theorem}
\label{T : VectorCV}Let $q\geq 2$ be fixed, and suppose that Condition 1 (or
2) of Corollary \ref{C : PunctualCLT} is satisfied.

\begin{enumerate}
\item For every $k\geq 1$, every $x_{1},...,x_{k}\in \mathbb{S}^{2}$ and
every compact subset $M\subset \mathbb{R}^{k}$,%
\begin{equation}
\sup_{\left( \lambda _{1},...,\lambda _{k}\right) \in M}\left\vert \mathbb{E}%
\left[ e^{\mathrm{i}\sum_{j=1}^{k}\lambda _{j}\overline{T}_{l}^{\left(
q\right) }\left( x_{j}\right) }\right] -\mathbb{E}\left[ ^{^{^{^{^{{}}}}}}e^{%
\mathrm{i}\sum_{j=1}^{k}\lambda _{j}\overline{T}_{l}^{{}}\left( x_{j}\right)
}\right] \right\vert \underset{l\rightarrow +\infty }{\longrightarrow }0.
\label{THcv}
\end{equation}

\item Fix $x_{1},...,x_{k}$ and denote by $\mathcal{L}\left( \overline{T}%
_{l}^{\left( q\right) };x_{1},...,x_{k}\right) $ and $\mathcal{L}\left(
\overline{T}_{l};x_{1},...,x_{k}\right) $ ($l\geq 1$), respectively, the law
of $\left( \overline{T}_{l}^{\left( q\right) }\left( x_{1}\right) ,...,%
\overline{T}_{l}^{\left( q\right) }\left( x_{k}\right) \right) $
and the law of $\left( \overline{T}_{l}\left( x_{1}\right)
,...,\overline{T}_{l}\left( x_{k}\right) \right) $. For every
metric $\gamma \left( \cdot ,\cdot \right) $ on $\mathbf{P}(
\mathbb{R}^{k}) $ such that $\gamma \left( \cdot ,\cdot \right) $
metrizes the weak convergence, it holds that
\begin{equation*}
\lim_{l\rightarrow +\infty }\gamma \left( \mathcal{L}\left( \overline{T}%
_{l}^{\left( q\right) };x_{1},...,x_{k}\right) ,\mathcal{L}\left( \overline{T%
}_{l};x_{1},...,x_{k}\right) \right) =0.
\end{equation*}
\end{enumerate}
\end{theorem}

\begin{proof}
The crucial point is that the spherical\ field $x\mapsto $ $\overline{T}%
_{l}^{\left( q\right) }\left( x\right) $ lives in the $q$th Wiener
chaos associated with the Gaussian space generated by $T$. By
using this fact, and by arguing as in the proof of Theorem
\ref{teo1}, one can show that the
vector $(\overline{T}_{l}^{\left( q\right) }\left( x_{1}\right) ,...,%
\overline{T}_{l}^{\left( q\right) }\left( x_{k}\right) )$ is
indeed equal in law to a vector of multiple Wiener-It\^{o}
integrals, of order $q$, with respect to a Brownian motion. Since
each element of this vector converges in law to a standard
Gaussian random variable, one can directly apply Theorem 1 and Proposition 2 in \cite%
{Pesco} to achieve the desired conclusion (see also \cite[Proposition 5]%
{Pesco}).
\end{proof}

\subsection{General subordination\label{SS : General Sub}}

We now give a solution to Problem (\textbf{P-III}), as stated at the end of
Section \ref{S : GaussSub}, where $F$ is a general real-valued function
belonging to the class $L_{0}^{2}\left( \mathbb{R},e^{-x^{2}/2}dx\right) $.
The function $F$ admits a unique representation of the form
\begin{equation}
F\left( z\right) =\sum_{q=1}^{\infty }\frac{c_{q}\left( F\right) }{q!}%
H_{q}\left( z\right) \text{, \ \ }z\in \mathbb{R}\text{,}  \label{devF}
\end{equation}%
where the Hermite polynomials $H_{q}$ are given by (\ref{Her}) and the real
coefficients $c_{q}\left( F\right) $, $q=1,2...$, are such that
\begin{equation}
\Sigma _{q}\frac{c_{q}\left( F\right) ^{2}}{q!}<+\infty \text{ .}
\label{proprCF}
\end{equation}

As a consequence, for every $l\geq 0$, the frequency component $F\left[ T%
\right] _{l}\left( x\right) $ defined in (\ref{subProj}) admits the expansion%
\begin{equation}
F\left[ T\right] _{l}\left( x\right) =\sum_{q=1}^{\infty }\frac{c_{q}\left(
F\right) }{q!}T_{l}^{\left( q\right) }\left( x\right) \text{, \ \ }x\in
\mathbb{S}^{2}\text{,}  \label{dev coeff}
\end{equation}%
where the series converges in $L^{2}\left( \mathbb{P}\right) $ for every
fixed $x$. Formula (\ref{dev coeff}) combined with Lemma \ref{L : Pl} yields
also that
\begin{equation*}
\mathbb{E(}F\left[ T\right] _{l}\left( x\right) F\left[ T\right] _{l}\left(
y\right) )=\frac{2l+1}{4\pi }P_{l}\left( \cos \left\langle x,y\right\rangle
\right) \sum_{q=1}^{\infty }\left( \frac{c_{q}\left( F\right) }{q!}\right)
^{2}\widetilde{C}_{l}^{(q)}\text{,}
\end{equation*}%
where $\widetilde{C}_{l}^{(q)}$ is given by (\ref{Short3}) or, equivalently,
by (\ref{VAR2}). The next result characterizes the asymptotic Gaussianity of
$F$-subordinated spherical random fields. Recall the definition of $%
\overline{F\left[ T\right] }_{l}$ given in (\ref{SubNorm}). The proof is
standard, and therefore omitted (it can be obtained e.g. along the lines of
\cite[Th.\ 4]{HuNu}).

\begin{theorem}
\label{T : GenCLT}Suppose that the following relations hold

\begin{enumerate}
\item For every $q\geq 1$, $\lim_{l\rightarrow +\infty }\frac{2l+1}{4\pi }%
\left( \frac{c_{q}\left( F\right) }{q!}\right) ^{2}\widetilde{C}_{l}^{(q)}/%
\mathbb{E(}F\left[ T\right] _{l}\left( x\right) ^{2})$ $\rightarrow $ $%
\sigma _{q}^{2}\in \lbrack 0,+\infty );$

\item $\sum_{m\geq 1}\left\{ c_{q}\left( F\right) /q!\right\} ^{2}\sigma
_{q}^{2}$ $\triangleq $ $\sigma ^{2}\left( F\right) $ $<$ $+\infty ;$

\item For every $q\geq 2$, the asymptotic relation (\ref{as11}) takes place
for $m=0$ and for every $p=\frac{q-1}{2}+1,...,q-1$, if $q-1$ is even, and
every $p=q/2,...,q-1$ if $q-1$ is odd;

\item $\lim_{p\rightarrow +\infty }\overline{\lim }_{l}\left( 2l+1\right)
\sum_{q=p+1}^{\infty }\left( \frac{c_{q}\left( F\right) }{q!}\right) ^{2}%
\widetilde{C}_{l}^{(q)}=0.$
\end{enumerate}

Then, for every $k\geq 1$, every $x_{1},...,x_{k}\in \mathbb{S}^{2}$ and
every compact $M\subset \mathbb{R}^{k}$,%
\begin{equation*}
\sup_{\left( \lambda _{1},...,\lambda _{k}\right) \in M}\left\vert \mathbb{E}%
\left[ e^{\mathrm{i}\sum_{j=1}^{k}\lambda _{j}\overline{F\left[ T\right] }%
_{l}\left( x_{j}\right) }\right]\! -\!\mathbb{E}\left[
e^{\mathrm{i}\sigma ^{2}\left( F\right)
^{1/2}\sum_{j=1}^{k}\lambda _{j}\overline{T}_{l}\left(
x_{j}\right) }\right] \right\vert \underset{l\rightarrow +\infty }{%
\!\rightarrow \!}0\text{,}
\end{equation*}%
where we used the notation (\ref{ti_barra}). In particular, the
last asymptotic relation implies that, for every
$\gamma(\cdot,\cdot)$ metrizing the weak convergence on
$\mathbf{P}( \mathbb{R}^{k}) $, the $\gamma$--distance between
$$(\overline{F\left[ T\right]} _{l}\left(
x_{1}\right),...,\overline{F\left[ T\right]} _{l}\left(
x_{k}\right))$$ and $\sigma ^{2}\left( F\right)^{1/2}(
\overline{T}_{l}\left( x_{1}\right),$ $...,$
$\overline{T}_{l}\left( x_{k}\right))$ converges to zero as
$l\rightarrow +\infty$.
\end{theorem}

\textbf{Remark. }A sufficient condition, ensuring that points 1 and 3 in the
statement of Theorem \ref{T : GenCLT} are verified, is the following: there
exist constants $\rho \left( q\right) >0$ such that (a) $\left( 2l+1\right)
\widetilde{C}_{l}^{(q)}\leq \rho \left( q\right) $ for every $q\geq 1$ and
every $l$, and (b) $\sum_{q=1}^{\infty }\left( \frac{c_{q}\left( F\right) }{%
q!}\right) ^{2}\rho \left( q\right) <+\infty $.

\section{Explicit sufficient conditions: convolutions and random walks\label%
{S : Explicit}}

In this section, we further explicit the conditions for the CLTs proved in
Section \ref{S: CLT} for the (Hermite) frequency components $T_{l}^{\left(
q\right) }$, $l\geq 0$. In particular, we shall establish sufficient
conditions that are more directly linked to primitive assumptions on the
behaviour of the angular power spectrum $\left\{ C_{l}:l\geq 0\right\} $.
The results of Section \ref{SS : Q2} and Section \ref{SS : Q3} cover,
respectively, the case $q=2$ and $q=3$. Section \ref{SS : CONJQ} contains
some partial findings for the case of a general $q$, as well as several
conjectures. These results will be used in Section \ref{S : Ang PS} to
deduce explicit conditions on the rate of decay of the angular power
spectrum $\left\{ C_{l}:l\geq 0\right\} $.

Our analysis is inspired by the following result, which is a particular case
of the statements contained in \cite[Section 3]{MaPe}, concerning fields on
Abelian groups. Consider indeed a centered real-valued Gaussian field $%
V=\left\{ V\left( \theta \right) :\theta \in \mathbb{T}\right\} $ defined on
the torus $\mathbb{T=[}0,2\mathbb{\pi )}$ (that we regard as an Abelian
compact group with group operation given by $xy=\left( x+y\right) \mathbf{mod%
}(2\pi )$). We suppose that the law of $V$ is \textsl{isotropic}, i.e. that $%
V\left( \theta \right) \overset{law}{=}V\left( x\theta \right) $ (in the
sense of stochastic processes) for every $x\in \mathbb{T}$, and also $%
\mathbb{E}V\left( \theta \right) ^{2}=1.$ We denote by $V\left( \theta
\right) =\sum_{l\in \mathbb{Z}}a_{l}e^{\mathrm{i}l\theta }$ the Fourier
decomposition of $V$, and we write $\Gamma _{l}^{V}=\mathbb{E}\left\vert
a_{l}\right\vert ^{2}$ (note that $\Gamma _{l}^{V}=\Gamma _{-l}^{V}$). Fix $%
q\geq 2$, and consider the Hermite-subordinated field $H_{q}\left[ V\right]
\left( \theta \right) $ $=H_{q}\left( V\left( \theta \right) \right) $,
where $q$ is the $q$th Hermite polynomial. The Fourier decomposition of $%
H_{q}\left[ V\right] $ is $H_{q}\left[ V\right] \left( \theta \right) $ $=$ $%
\sum_{l\in \mathbb{Z}}a_{l}^{\left( q\right) }e^{\mathrm{i}l\theta }$. We
write $N,N^{\prime }$ to indicate a pair of independent centered Gaussian
random variables with common variance equal to $1/2$: in \cite{MaPe} it is
proved that to have the \textsl{high-frequency CLT}%
\begin{equation}
\frac{a_{l}^{\left( q\right) }}{Var\left( a_{l}^{\left( q\right) }\right)
^{1/2}}=\frac{\int_{\mathbb{T}}H_{q}\left[ V\right] \left( \theta \right)
e^{-\mathrm{i}l\theta }d\theta }{Var\left( a_{l}^{\left( q\right) }\right)
^{1/2}}\underset{l\rightarrow \infty }{\overset{law}{\rightarrow }}N+%
\mathrm{i}N^{\prime }  \label{clll}
\end{equation}%
it is \textsl{necessary and sufficient} that, for every $p=1,...,q-1$,
\begin{equation}
\lim_{l\rightarrow +\infty }\sup_{j\in \mathbb{Z}}\mathbb{P}\left[
U_{p}=j\mid U_{q}=l\right] =0\text{,}  \label{rwAbel}
\end{equation}%
where $\left\{ U_{n}:n\geq 0\right\} $ is the random walk on $\mathbb{Z}$
whose law is given by $U_{0}=0$ and $$\mathbb{P}\left[ U_{n+1}=j\mid U_{n}=k%
\right] =\Gamma _{j-k}^{V}.$$ Note that the law of the random variable $%
U_{n} $ has trivially the form of a \textsl{convolution }of the
coefficients $\Gamma _{l}^{V}$ (see also the discussion below).
The correspondence between (\ref{clll}) and the \textquotedblleft
random walk bridge\textquotedblright\ (\ref{rwAbel}) has been used
in \cite{MaPe} to establish explicit conditions on the power
spectrum $\{ \Gamma _{l}^{V}\} $ to have that (\ref{clll}) holds.

In what follows, we shall unveil (and apply) an analogous connection between
the CLTs proved in Section \ref{S: CLT} and some specific convolutions and
random walks on $\widehat{SO\left( 3\right) }$.

\subsection{Convolutions on $\widehat{SO\left( 3\right) }$}

In the light of Part 3 of Theorem \ref{teo1} and by Corollary \ref{C :
PunctualCLT}, we will focus on the sequence $\left\{ a_{l0;q}:l\geq
0\right\} $ (see (\ref{Short1.5})), whose behaviour as $l\rightarrow +\infty
$ yields an asymptotic characterization of the fields $T_{l}^{\left(
q\right) }\left( \cdot \right) $ defined in (\ref{Short2}). A crucial point
is the simple fact that the numerator of (\ref{as11}), for $m=0$, can be
developed as a multiple sum involving products of four generalized Gaunt
integrals, so that, by (\ref{megagaunt}), the asymptotic expressions
appearing in Theorem \ref{teo1} can be studied by means of the properties of
linear combinations of products of Clebsch-Gordan coefficients. As
anticipated, a very efficient tool for our analysis will be the use of
convolutions on $\mathbb{N}$, that we endow with an hypergroup structure
isomorphic to $\widehat{SO\left( 3\right) }$, i.e. the dual of $SO\left(
3\right) $. This will be the object of the subsequent discussion.

From now on, and for the rest of the section, we shall fix a sequence $%
\left\{ C_{l}:l\geq 0\right\} $, representing the angular power spectrum of
an isotropic centered, normalized Gaussian field $T$ over $\mathbb{S}^{2}$,
as in Section \ref{S : GaussSub}. Whenever convenient we shall write
\begin{equation}
\Gamma _{l}\triangleq (2l+1)C_{l}\text{ , \ }l\geq 0\text{,}
\label{FreqSpectrum}
\end{equation}%
so that, for $l\geq 1$ and up to the constant $1/4\pi $, the parameter $%
\Gamma _{l}$ represents the variance of the projection of the Gaussian field
$T$ in (\ref{specrap}) on the frequency $l$: indeed, according to Lemma \ref%
{L : Pl}, $Var(T_{l})=\Gamma _{l}/4\pi $. Also, we define the following
\textsl{convolutions} of the coefficients $\Gamma _{l}$ (in the following
expressions, the sums over indices $l_{i}$, $L_{i}$ ... range implicitly
from $0$ to $+\infty $):%
\begin{align}
& \!\!\!\!\!\!\!\!\widehat{\Gamma }_{2,l}\!=\!\sum_{l_{1},l_{2}}\Gamma
_{l_{1}}\Gamma _{l_{2}}(C_{l_{1}0l_{2}0}^{l0})^{2}\text{ ,}\qquad \qquad
\qquad \qquad \qquad \qquad  \label{cgconv-2} \\
& \!\!\!\!\!\!\!\!\widehat{\Gamma }_{3,l}\!=\!\sum_{L_{1},l_{3}}\!\widehat{%
\Gamma }_{2,L_{1}}\Gamma
_{l_{3}}(C_{L_{1}0l_{3}0}^{l0})^{2}\!=\!\sum_{l_{1},l_{2},l_{3}}\!\Gamma
_{l_{1}}\Gamma _{l_{2}}\Gamma
_{l_{3}}\sum_{L_{1}}(C_{l_{1}0l_{2}0l_{3}0}^{L_{1}l;0})^{2}\text{, ...}
\label{cgconv-1}
\end{align}%
$\ $ \\[-30pt]
\begin{equation}
\widehat{\Gamma }_{q,l}\!=\!\sum_{L_{1},l_{q}}\!\widehat{\Gamma }%
_{q-1,L_{q-1}}\!\Gamma
_{l_{q}}\!(C_{L_{q-1}0l_{q}0}^{l0})^{2}\!=\!\sum_{l_{1}...l_{q}}\!\Gamma
_{l_{1}}\!...\!\Gamma
_{l_{q}}\!\sum_{L_{1}\!...\!L_{q-2}}%
\!(C_{l_{1}0...l_{q}0}^{L_{1}...L_{q-2}l;0})^{2}.  \label{cgconv}
\end{equation}%
We stress that the equalities in formulae (\ref{cgconv-1}) and (\ref{cgconv}%
) are consequences of (\ref{0_conv}). It will be also convenient to define a
\textsl{*-convolution} of order $p\geq 2$ as:%
\begin{eqnarray}
\widehat{\Gamma }_{p,l;l_{1}}^{\ast } &=&\sum_{l_{2}}\cdot \cdot \cdot
\sum_{l_{p}}\Gamma _{l_{2}}\cdot \cdot \cdot \Gamma
_{l_{p}}\sum_{L_{1}...L_{p-2}}\left\{
C_{l_{1}0l_{2}0}^{L_{1}0}C_{L_{1}0l_{3}0}^{L_{2}0}...C_{L_{p-2}0l_{p}0}^{l0}%
\right\} ^{2}  \notag \\
&=&\sum_{l_{2}}\cdot \cdot \cdot \sum_{l_{p}}\Gamma _{l_{2}}\cdot \cdot
\cdot \Gamma _{l_{p}}\sum_{L_{1}...L_{p-2}}\left\{
C_{l_{1}0l_{2}0...l_{p}0}^{L_{1}...l;0}\right\} ^{2}.  \label{starconv}
\end{eqnarray}%
Note that the number of sums following the equalities in formula
(\ref{starconv}) is $p-1$: however, we choose to keep the symbol
$p$ to denote *-convolutions, since it
is consistent with the probabilistic representations given in formulae (\ref%
{pint1}) and (\ref{pint2}) below. The above *-convolution has the following
property: for every $p=2,...,q$%
\begin{equation*}
\sum_{l_{1}}\widehat{\Gamma }_{q+1-p,l_{1}}\widehat{\Gamma }%
_{p,l;l_{1}}^{\ast }=\widehat{\Gamma }_{q,l}\text{ , and, in particular, }%
\sum_{l_{1}}\Gamma _{l_{1}}\widehat{\Gamma }_{q,l;l_{1}}^{\ast }=\widehat{%
\Gamma }_{q,l}\text{ .}
\end{equation*}%
The *-convolution of order 2 can be written more explicitly as
\begin{equation}
\widehat{\Gamma }_{2,l;l_{1}}^{\ast }=\sum_{l_{2}}\Gamma
_{l_{2}}(C_{l_{1}0l_{2}0}^{l0})^{2}.  \label{cgconv1}
\end{equation}

\textbf{Remarks. (1) }(\textit{Probabilistic interpretation of the
convolutions}) Write first $\Gamma _{\ast }\triangleq \sum_{l}\Gamma _{l}$
(plainly, in our framework $\Gamma _{\ast }=4\pi $, but the following
discussion applies to coefficients $\left\{ \Gamma _{l}\right\} $ such that $%
\Gamma _{\ast }>0$ is arbitrary) so that $l\longmapsto \Gamma _{l}/\Gamma
_{\ast }$ defines a probability on $\mathbb{N}$. The second orthonormality
relation in (\ref{orto1}) implies that, for fixed $l_{1},l_{2}$, the
application $l\longmapsto (C_{l_{1}0l_{2}0}^{l0})^{2}$ is a probability on $%
\mathbb{N}$. Now define the law of a (homogeneous) \textsl{Markov chain }$%
\left\{ Z_{n}:n\geq 1\right\} $ as follows:
\begin{align}
\mathbb{P}\left\{ Z_{1}=l\right\} & =\Gamma _{l}/\Gamma _{\ast }  \label{rw1}
\\
\mathbb{P}\left\{ Z_{n+1}=l\mid Z_{n}=L\right\} & =\sum_{l_{0}}\frac{\Gamma
_{l_{0}}}{\Gamma _{\ast }}\left( C_{l_{0}0L0}^{l0}\right) ^{2}\text{.}
\label{rw2}
\end{align}%
It is clear that $\mathbb{P}\left\{ Z_{q}=l\right\} =\widehat{\Gamma }%
_{q,l}/\left( \Gamma _{\ast }\right) ^{q}$, and also, for $p\geq 2$,%
\begin{eqnarray}
\frac{\widehat{\Gamma }_{p,l:l_{1}}^{\ast }}{\left( \Gamma _{\ast }\right)
^{p-1}} &=&\mathbb{P}\left\{ Z_{p}=l\mid Z_{1}=l_{1}\right\}  \label{pint1}
\\
\frac{\widehat{\Gamma }_{p,l:l_{1}}^{\ast }\widehat{\Gamma }_{q+1-p,l_{1}}}{%
\left( \Gamma _{\ast }\right) ^{q}} &=&\mathbb{P}\left\{ \left(
Z_{q}=l\right) \cap \left( Z_{q+1-p}=l_{1}\right) \right\} \text{ \ \ (}q>p-1%
\text{).}  \label{pint2}
\end{eqnarray}%
The following quantity will be crucial in the subsequent sections:%
\begin{equation}
\frac{\widehat{\Gamma }_{q+1-p,l;\lambda }^{\ast }\widehat{\Gamma }%
_{p,\lambda }}{\sum_{L}\widehat{\Gamma }_{p,L}\widehat{\Gamma }%
_{q+1-p,l;L}^{\ast }}=\frac{\widehat{\Gamma }_{q+1-p,l;\lambda }^{\ast }%
\widehat{\Gamma }_{p,\lambda }}{\widehat{\Gamma }_{q,l}}=\mathbb{P}\left\{
Z_{p}=\lambda \mid Z_{q}=l\right\} \text{ \ (}q>p\text{);}  \label{impint}
\end{equation}%
observe that the last relation in (\ref{impint}) derives from $$\widehat{%
\Gamma }_{q+1-p,l;\lambda }^{\ast }/\left( \Gamma _{\ast }\right) ^{q-p}=%
\mathbb{P\{}\left( Z_{q+1-p}=l\right)  | \left( Z_{1}=\lambda
\right) \} = \mathbb{P}\left\{ \left( Z_{q}=\lambda \right) |
\left( Z_{p}=l\right) \right\},$$where the last equality is a
consequence of the
homogeneity of $Z$. Note also that we can identify each natural number $%
l\geq 0$ with an irreducible representation of $SO\left( 3\right) $. It
follows that the formal addition $l_{1}+l_{2}\triangleq
\sum_{l}l(C_{l_{1}0l_{2}0}^{l0})^{2}$ may be used to endow $\widehat{%
SO\left( 3\right) }$ with an hypergroup structure. In this sense, we can
interpret the chain $\left\{ Z_{n}:n\geq 1\right\} $ as a \textsl{random walk%
} on the hypergroup $\widehat{SO\left( 3\right) }$, in a spirit similar to
\cite{GKR}. In Section \ref{S : Roynette}, we will discuss a physical
interpretation of these convolutions and establish a precise connection
between the objects introduced in this section and the notion of convolution
appearing in \cite{GKR}.

(\textbf{2}) (\textit{A comparison with the Abelian case}) In
\cite{MaPe}, where we dealt with similar problems in the case of
homogenous spaces of Abelian groups, we used extensively
convolutions over $\mathbb{Z}$. This kind of convolutions, that we
note $_{A}\widehat{\Gamma }_{q,l}$ ($q\geq 2$, $l\in \mathbb{Z}$)
are obtained as in (\ref{cgconv-2})-(\ref{cgconv1}), by taking
sums over $\mathbb{Z}$ (instead than over $\mathbb{N}$) and by
replacing the Clebsch-Gordan symbols $( C_{l_{1}0l_{2}0}^{l0} )
^{2}$ with the indicator $\mathbf{1}_{l_{1}+l_{2}=l}$. Note that
these indicator functions do indeed provide the Clebsch-Gordan
coefficients associated with the irreducible representations of
the $1$-dimensional torus $\mathbb{T=[}0,2\mathbb{\pi )}$,
regarded as a compact Abelian group with group operation
$xy=\left( x+y\right) \left( \mathbf{mod}(2\pi )\right) $
(this is equivalent to the trivial relation $e^{\mathrm{i}l_{1}x}e^{\mathrm{i%
}l_{2}x}=\sum_{l}\mathbf{1}_{l_{1}+l_{2}=l}e^{\text{i}lx}=e^{\text{i}\left(
l_{1}+l_{2}\right) x}$)$.$ Note also that in the Abelian case one has $_{A}%
\widehat{\Gamma }_{p,l;l_{1}}^{\ast }$ $=$ $_{A}\widehat{\Gamma }%
_{p,l-l_{1}}.$ Also, if $\Gamma _{l}=\Gamma _{l}^{V}$, where $\{
\Gamma
_{l}^{V}\} $ is the power spectrum of the Gaussian field $V$ on $%
\mathbb{T}$ appearing in (\ref{clll}), one has that
$_{A}\widehat{\Gamma _{q,l}^{V}}$ $=$ $\mathbb{P}\left[
U_{q}=l\right] $, where $\left\{ U_{n}\right\} $ is the random
walk given in (\ref{rwAbel}).

\subsection{The case $q=2$\label{SS : Q2}}

In this subsection, we provide a sufficient condition on the spectrum $%
\left\{ C_{l}:l\geq 0\right\} $ (or, equivalently, on $\left\{ \Gamma
_{l}:l\geq 0\right\} $, as defined in (\ref{FreqSpectrum})) to have the CLT (%
\ref{as00}) in the quadratic case $q=2$. This condition is stated in
Proposition \ref{cglem}, and is obtained via some preliminary (technical)
computations and lemmas.

According to Part 3 of Theorem \ref{teo1}, to deal with (\ref{as00}) we
shall find sufficient conditions to have that (\ref{as11}) takes place for $%
m=0$, $q=2$ and $p=1$. From (\ref{VAR2}) we deduce
\begin{equation}
\widetilde{C}_{l}^{\left( 2\right) }=2\left\{ \sum_{l_{1},l_{2}=0}^{\infty }%
\frac{(2l_{1}+1)(2l_{2}+1)}{4\pi (2l+1)}%
C_{l_{1}}C_{l_{2}}(C_{l_{1}0l_{2}0}^{l0})^{2}\right\} ^{2}.  \label{den2}
\end{equation}%
On the other hand, the multiple sums appearing in the numerator of (\ref%
{as11}) become ($q=2$, $p=1$)%
\begin{eqnarray}
&&\sum_{j_{1},n_{1},j_{2},n_{2}}\!\!\!C_{j_{1}}\!C_{j_{2}}\!\left\vert
\sum_{l_{1},m_{1}}C_{l_{1}}\!\mathcal{G}\!\left\{
l_{1},m_{1};j_{1},n_{1};l,-m\right\} \!\mathcal{G}\!\left\{
l_{1},m_{1};j_{2},n_{2};l,-m\right\} \right\vert ^{2}  \notag \\
&=&\! \! \frac{1}{\left[ 4\pi (2l+1)\right] ^{2}}%
\sum_{j_{1},n_{1},j_{2},n_{2}}\! \Gamma _{j_{1}}\Gamma _{j_{2}}\!\left\vert
\sum_{l_{1},m_{1}}\Gamma
_{l_{1}}C_{l_{1}m_{1}j_{1}n_{1}}^{lm}C_{l_{1}0j_{1}0}^{l0}C_{l_{1}m_{1}j_{2}n_{2}}^{lm}C_{l_{1}0j_{2}0}^{l0}\right\vert ^{2}
\notag \\
&=&\frac{1}{\left[ 4\pi (2l+1)\right] ^{2}}\sum_{j_{1},n_{1},j_{2}}\Gamma
_{j_{1}}\Gamma _{j_{2}}\left\vert \sum_{l_{1},m_{1}}\Gamma
_{l_{1}}C_{l_{1}m_{1}j_{1}n_{1}}^{lm}C_{l_{1}0j_{1}0}^{l0}C_{l_{1}m_{1}j_{2}n_{2}}^{lm}C_{l_{1}0j_{2}0}^{l0}\right\vert ^{2}
\notag \\
&=&\frac{1}{\left[ 4\pi (2l+1)\right] ^{2}}\sum_{j_{1}j_{2}}%
\sum_{l_{1}l_{2}}\Gamma _{j_{1}}\Gamma _{j_{2}}\Gamma _{l_{1}}\Gamma
_{l_{2}}C_{l_{1}0j_{1}0}^{l0}C_{l_{1}0j_{2}0}^{l0}C_{l_{2}0j_{1}0}^{l0}C_{l_{2}0j_{2}0}^{l0}
\notag \\
&&\text{\ \ \ \ \ \ \ }\times \left\{
\sum_{n_{1}n_{2}m_{1}m_{2}}C_{l_{1}m_{1}j_{1}n_{1}}^{lm}C_{l_{1}m_{1}j_{2}n_{2}}^{lm}
C_{l_{2}m_{2}j_{1}n_{1}}^{lm}C_{l_{2}m_{2}j_{2}n_{2}}^{lm}\right\}.
\label{num2}
\end{eqnarray}%
Now, from \cite[Eq. 8.7.4.20]{VMK} we deduce that%
\begin{eqnarray*}
&&\sum_{n_{1}n_{2}}%
\sum_{m_{1}m_{2}}C_{l_{1}m_{1}j_{1}n_{1}}^{lm}C_{l_{1}m_{1}j_{2}n_{2}}^{lm}C_{l_{2}m_{2}j_{1}n_{1}}^{lm}C_{l_{2}m_{2}j_{2}n_{2}}^{lm}
\\
&=&(-1)^{\beta }\sum_{s\sigma }(2s+1)(2l+1)(C_{lms\sigma }^{lm})^{2}\left\{
\begin{array}{ccc}
l_{1} & j_{1} & l \\
l & s & l_{2}%
\end{array}%
\right\} \left\{
\begin{array}{ccc}
l_{1} & j_{2} & l \\
l & s & l_{2}%
\end{array}%
\right\} \\
&=&\sum_{s}(2s+1)(2l+1)(C_{lms0}^{lm})^{2}\left\{
\begin{array}{ccc}
l_{1} & j_{1} & l \\
l & s & l_{2}%
\end{array}%
\right\} \left\{
\begin{array}{ccc}
l_{1} & j_{2} & l \\
l & s & l_{2}%
\end{array}%
\right\} \text{,}
\end{eqnarray*}%
where $\beta =l_{1}+j_{1}+l_{2}+j_{2}$, and we used the Wigner $6j$ symbols,
as defined in (\ref{wig6j}). The last equality follows because the quantity $%
l_{1}+j_{1}$ $+l_{2}+j_{2}$ $+2l$ must be necessarily even, and therefore $%
\beta $ must be even as well. It should be noted that the role of the pairs $%
(j_{1},n_{1})$ and $(l_{1},m_{1})$ is perfectly symmetric, so we obtain also
\begin{eqnarray*}
&&\sum_{n_{1}n_{2}}%
\sum_{m_{1}m_{2}}C_{l_{1}m_{1}j_{1}n_{1}}^{lm}C_{l_{1}m_{1}j_{2}n_{2}}^{lm}C_{l_{2}m_{2}j_{1}n_{1}}^{lm}C_{l_{2}m_{2}j_{2}n_{2}}^{lm}
\\
&=&\sum_{s}(2s+1)(2l+1)(C_{lms0}^{lm})^{2}\left\{
\begin{array}{ccc}
j_{1} & l_{1} & l \\
l & s & j_{2}%
\end{array}%
\right\} \left\{
\begin{array}{ccc}
j_{1} & l_{2} & l \\
l & s & j_{2}%
\end{array}%
\right\} \text{ ,}
\end{eqnarray*}%
whence
\begin{eqnarray}
&&\sum_{s}(2s+1)(2l+1)(C_{lms0}^{lm})^{2}\left\{
\begin{array}{ccc}
j_{1} & l_{1} & l \\
l & s & j_{2}%
\end{array}%
\right\} \left\{
\begin{array}{ccc}
j_{1} & l_{2} & l \\
l & s & j_{2}%
\end{array}%
\right\}  \label{pau1} \\
&\equiv &\sum_{s}(2s+1)(2l+1)(C_{lms0}^{lm})^{2}\left\{
\begin{array}{ccc}
l_{1} & j_{1} & l \\
l & s & l_{2}%
\end{array}%
\right\} \left\{
\begin{array}{ccc}
l_{1} & j_{2} & l \\
l & s & l_{2}%
\end{array}%
\right\} \text{ .}  \label{pau2}
\end{eqnarray}

\begin{lemma}
\label{lemma6j}For all integers $l,l_{1},l_{2},j_{1},j_{2}$ it holds that,
for some positive constant $c$,
\begin{eqnarray*}
&&\sum_{s}(2s+1)(2l+1)(C_{l0s0}^{l0})^{2}\left\{
\begin{array}{ccc}
l_{1} & j_{1} & l \\
l & s & l_{2}%
\end{array}%
\right\} \left\{
\begin{array}{ccc}
l_{1} & j_{2} & l \\
l & s & l_{2}%
\end{array}%
\right\} \\
&\leq &c\max \left[ \frac{1}{\sqrt[5]{2l_{1}+1}}\wedge \frac{1}{\sqrt[5]{%
2l_{2}+1}},\frac{1}{\sqrt[5]{2j_{1}+1}}\wedge \frac{1}{\sqrt[5]{2j_{2}+1}}%
\right] .
\end{eqnarray*}
\end{lemma}

\begin{proof}
Assume without loss of generality $j_{1},j_{2}>l_{1}$ otherwise we focus on
(\ref{pau2}) rather than (\ref{pau1}). For $\alpha \in (0,1),$ we have that%
\begin{equation*}
\sum_{s}(2s+1)(2l+1)(C_{l0s0}^{l0})^{2}\left\{
\begin{array}{ccc}
l_{1} & j_{1} & l \\
l & s & l_{2}%
\end{array}%
\right\} \left\{
\begin{array}{ccc}
l_{1} & j_{2} & l \\
l & s & l_{2}%
\end{array}%
\right\}
\end{equation*}%
\begin{eqnarray*}
&\leq &\sum_{s\leq l_{1}^{\alpha }}(2s+1)(2l+1)(C_{l0s0}^{l0})^{2}\left\{
\begin{array}{ccc}
l_{1} & j_{1} & l \\
l & s & l_{2}%
\end{array}%
\right\} \left\{
\begin{array}{ccc}
l_{1} & j_{2} & l \\
l & s & l_{2}%
\end{array}%
\right\} \\
&&+\sum_{s>l_{1}^{\alpha }}(2s+1)(2l+1)(C_{l0s0}^{l0})^{2}\left\{
\begin{array}{ccc}
l_{1} & j_{1} & l \\
l & s & l_{2}%
\end{array}%
\right\} \left\{
\begin{array}{ccc}
l_{1} & j_{2} & l \\
l & s & l_{2}%
\end{array}%
\right\}
\end{eqnarray*}\\[-30pt]
\begin{eqnarray*}
&\leq &Cl_{1}^{2\alpha }(2l+1)\max_{s\leq l_{1}^{\alpha }}\left\{
\begin{array}{ccc}
l_{1} & j_{1} & l \\
l & s & l_{2}%
\end{array}%
\right\} \left\{
\begin{array}{ccc}
l_{1} & j_{2} & l \\
l & s & l_{2}%
\end{array}%
\right\} \\
&&+\left\{ \max_{s>l_{1}^{\alpha }}(C_{l0s0}^{l0})^{2}\right\}
\sum_{s}(2s+1)(2l+1)\left\{
\begin{array}{ccc}
l_{1} & j_{1} & l \\
l & s & l_{2}%
\end{array}%
\right\} \left\{
\begin{array}{ccc}
l_{1} & j_{2} & l \\
l & s & l_{2}%
\end{array}%
\right\}
\end{eqnarray*}\\[-30pt]
\begin{eqnarray*}
\leq Cl_{1}^{2\alpha }(2l+1)\frac{1}{(2l+1)(2l_{1}+1)}+\frac{C}{%
l_{1}^{\alpha /2}}\frac{2l+1}{\sqrt{j_{1}j_{2}}}=O(l_{1}^{2\alpha
-1}+l_{1}^{-\alpha /2})=O(\frac{1}{\sqrt[5]{l_{1}}}),
\end{eqnarray*}%
where the last equality has been obtained by setting $\alpha = 2/5$.
The second last step follows because $j_{1},j_{2}\geq l_{1},l_{2}$ implies $%
j_{1},j_{2}>l/2,$ in view of the triangle inequalities $%
l_{1}+j_{1},l_{1}+j_{2}>l$; also, we used the inequality\emph{\ }$\left\{
\max_{s>l_{1}^{\alpha }}(C_{l0s0}^{l0})^{2}\right\} \leq l_{1}^{-\alpha /2}$%
, see Lemma \ref{cglem} below. The bound with $l_{2}$ can be obtained by
exploiting the symmetries of the $6j$ coefficients; in particular, we recall
that (see (\cite[Eq. 9.4.2.2]{VMK}))%
\begin{equation*}
\left\{
\begin{array}{ccc}
l_{1} & j_{1} & l \\
l & s & l_{2}%
\end{array}%
\right\} \equiv \left\{
\begin{array}{ccc}
l & j_{1} & l_{2} \\
l_{1} & s & l%
\end{array}%
\right\} \equiv \left\{
\begin{array}{ccc}
l_{2} & j_{1} & l \\
l & s & l_{1}%
\end{array}%
\right\} \text{ .}
\end{equation*}
\end{proof}

\textbf{Remark. }The bound provided in Lemma (\ref{lemma6j}) is
sufficient for our purposes below and we did not investigate its
efficiency in detail. We remark, however, by setting
$j_{1}=j_{2}=0$, we have explicitly (see \cite[Eq. 8.5.1.2]{VMK})
\begin{align}
\sum_{n_{1}n_{2}m_{1}m_{2}}C_{l_{1}m_{1}j_{1}n_{1}}^{lm}C_{l_{1}m_{1}j_{2}n_{2}}^{lm}C_{l_{2}m_{2}j_{1}n_{1}}^{lm}C_{l_{2}m_{2}j_{2}n_{2}}^{lm}
\notag \\
=\sum_{m_{1}m_{2}}C_{l_{1}m_{1}00}^{lm}C_{l_{1}m_{1}00}^{lm}C_{l_{2}m_{2}00}^{lm}C_{l_{2}m_{2}00}^{lm}\equiv 1%
\text{ .}  \notag
\end{align}%


\begin{lemma}
\label{cglem} As $l_{1}\rightarrow +\infty $, $C_{l0l_{1}0}^{l0}=O(\frac{1}{%
\sqrt[4]{l_{1}}})$. 
\end{lemma}

\begin{proof}
Unless the triangle condition $2l\geq l_{1}$ is satisfied, the
Clebsch-Gordan coefficient is identically zero and the bound is trivial. Now
recall that%
\begin{equation*}
C_{l0l_{1}0}^{l0}=\frac{\sqrt{2l+1}\left[ (2l+l_{1})/2\right] !}{\left[
l_{1}/2\right] !\left[ (2l-l_{1})/2\right] !\left[ l_{1}/2\right] !}\left\{
\frac{l_{1}!(2l-l_{1})!l_{1}!}{(2l+l_{1}+1)!}\right\} ^{1/2}.
\end{equation*}%
For sequences $\left\{ a_{l}\right\} $ and $\left\{ b_{l}\right\} $, write $%
a_{l}\approx b_{l}$ when both $a_{l}=O(b_{l})$ and $b_{l}=O(a_{l})$ hold
true. From Stirling's formula%
\begin{align*}
C_{l0l_{1}0}^{l0} \!&\approx \!\frac{\sqrt{2l+1}\left[
(2l+l_{1})/2\right]
^{(2l+l_{1})/2+1/2}}{\left[ l_{1}/2\right] ^{l_{1}+1}\left[ (2l-l_{1}+1)/2%
\right] ^{(2l-l_{1})/2+1/2}}\!\left\{ \frac{%
l_{1}^{2l_{1}+1}(2l-l_{1})^{(2l-l_{1})+1/2}}{(2l+l_{1}+1)^{2l+l_{1}+3/2}}%
\right\} ^{1/2} \\
&=\!\frac{\sqrt{2l+1}(2l+l_{1})^{(2l+l_{1})/2+1/2}}{%
l_{1}^{l_{1}+1}(2l-l_{1}+1)^{(2l-l_{1})/2+1/2}}\left\{ \frac{%
l_{1}^{2l_{1}+1}(2l-l_{1})^{(2l-l_{1})+1/2}}{(2l+l_{1}+1)^{2l+l_{1}+3/2}}%
\right\} ^{1/2} \\
&=\!\!\frac{\sqrt{2l+1}}{l_{1}^{1/2}(2l\!-l_{1}\!+\!1)^{1/4}}\frac{1}{%
(2l\!+\!l_{1}\!+\!1)^{1/4}}\!\leq\! \frac{\sqrt[4]{2l+1}}{l_{1}^{1/2}(2l\!-\!l_{1}\!+\!1)^{1/4}}%
\!=\!O(\frac{1}{\sqrt[4]{l_{1}}})
\end{align*}
\end{proof}

\noindent We can finally state a sufficient condition for the CLT
(\ref{as00}) in the case $q=2.$

\begin{proposition}
\label{lemmaq2}\ For $q=2$, a sufficient condition for the CLT (\ref{as00})
is the following asymptotic relation%
\begin{equation}
\lim_{l\rightarrow +\infty }\sup_{l_{1}}\frac{\sum_{l_1}\Gamma
_{l_{1}}\Gamma _{l_{2}}\left\{ C_{l_{1}0l_{2}0}^{l0}\right\} ^{2}}{%
\sum_{l_{1},l_{2}}\Gamma _{l_{1}}\Gamma _{l_{2}}(C_{l_{1}0l_{2}0}^{l0})^{2}}%
\!=\!\lim_{l\rightarrow +\infty }\sup_{l_{1}}\mathbb{P}\left\{
Z_{1}\!=\!l_{1}\!\mid \! Z_{2} \! = \!l_{2}\right\} \!=0\text{,}
\label{sufcon2}
\end{equation}%
where the $\left\{ \Gamma _{l}\right\} $ are given by (\ref{FreqSpectrum})
and $\left\{ Z_{l}\right\} $ is the Markov chain defined in formulae (\ref%
{rw1}) and (\ref{rw2}).
\end{proposition}

\begin{proof}
In the sequel, we shall use repeatedly the trivial inequality
\begin{equation}
\sum_{j=0}^{n}\left| \frac{c_{j}}{a_{j}\wedge b_{j}}\right|
=\sum_{j:a_{j}\leq b_{j}}\left| \frac{c_{j}}{a_{j}}\right|
+\sum_{j:a_{j}>b_{j}}\left| \frac{c_{j}}{b_{j}}\right| \leq
\sum_{j=0}^{n}\left| \frac{c_{j}}{a_{j}}\right| +\sum_{j=0}^{n}\left| \frac{%
c_{j}}{b_{j}}\right| \text{,}  \label{trivine}
\end{equation}%
which holds for arbitrary $n$ and real vectors $\left\{ a_{j}\right\} $, $%
\left\{ b_{j}\right\} $ and $\left\{ c_{j}\right\} $. In view of Lemma \ref%
{lemma6j}, by using a generalized Cauchy-Schwartz inequality, (\ref{trivine}%
) and symmetry considerations, we obtain that the expression
(\ref{num2}) is such that
\begin{eqnarray*}
(\ref{num2}) &\leq &\frac{1}{[4\pi (2l+1)]^{2}}\sum_{j_{1},j_{2}}%
\sum_{l_{1},l_{2}}\Gamma _{j_{1}}\Gamma _{j_{2}}\Gamma _{l_{1}}\Gamma
_{l_{2}}\left|
C_{l_{1}0j_{1}0}^{l0}C_{l_{1}0j_{2}0}^{l0}C_{l_{2}0j_{1}0}^{l0}C_{l_{2}0j_{2}0}^{l0}\right|
\\
&\leq &\frac{2}{[4\pi
(2l+1)]^{2}}\!\sum_{j_{1},j_{2}}\!\sum_{l_{1},l_{2}}\Gamma
_{j_{1}}\Gamma _{j_{2}}\Gamma _{l_{1}}\Gamma _{l_{2}}\!\left|
C_{l_{1}0j_{1}0}^{l0}C_{l_{1}0j_{2}0}^{l0}C_{l_{2}0j_{1}0}^{l0}C_{l_{2}0j_{2}0}^{l0}\right|
\frac{1}{\sqrt[5]{j_{1}}}\\
&\leq &\frac{1}{8[\pi
(2l+1)]^{2}}\sqrt{\sum_{l_{1}j_{1}}\frac{\Gamma _{l_{1}}\Gamma
_{j_{1}}}{\sqrt[5]{j_{1}^{2}}}\left\{
C_{l_{1}0j_{1}0}^{l0}\right\} ^{2}\sum_{l_{1}j_{2}}\Gamma
_{l_{1}}\Gamma
_{j_{2}}\left\{ C_{l_{1}0j_{2}0}^{l0}\right\} ^{2}} \\
&&\times \sqrt{\sum_{l_{2}j_{1}}\Gamma _{l_{2}}\Gamma
_{j_{1}}\left\{ C_{l_{2}0j_{1}0}^{l0}\right\}
^{2}\sum_{l_{2}j_{2}}\Gamma _{l_{2}}\Gamma _{j_{2}}\left\{
C_{l_{2}0j_{2}0}^{l0}\right\} ^{2}}.
\end{eqnarray*}%
The last expression is less than
\begin{equation}%
\frac{ \sum_{l_{1}l_{2}}\Gamma _{l_{1}}\Gamma _{l_{2}}\left\{
C_{l_{1}0l_{2}0}^{l0}\right\} ^{2}}{8[\pi (2l+1)]^{2}}
\sqrt{\sum_{l_{1}j_{1}}\frac{\Gamma _{l_{1}}\Gamma
_{j_{1}}}{\sqrt[5]{j_{1}^{2}}}\left\{
C_{l_{1}0j_{1}0}^{l0}\right\} ^{2}\sum_{l_{1}j_{2}}\Gamma
_{l_{1}}\Gamma _{j_{2}}\left\{ C_{l_{1}0j_{2}0}^{l0}\right\}
^{2}}. \label{urca}
\end{equation}%
Now
\begin{eqnarray*}
\sum_{l_{1}j_{1}}\frac{\Gamma _{l_{1}}\Gamma _{j_{1}}}{\sqrt[5]{j_{1}^{2}}}%
\left\{ C_{l_{1}0j_{1}0}^{l0}\right\} ^{2} &\leq &j^{\ast }\max_{j_{1}\leq
j^{\ast }}\left[ \sum_{l_{1}\geq 0}\Gamma _{l_{1}}\Gamma _{j_{1}}\left\{
C_{l_{1}0j_{1}0}^{l0}\right\} ^{2}\right] \\
&&+\frac{1}{\sqrt[5]{(j^{\ast })^{2}}}\sum_{l_{1},j_{1}\geq 0}\Gamma
_{l_{1}}\Gamma _{j_{1}}\left\{ C_{l_{1}0j_{1}0}^{l0}\right\} ^{2}.
\end{eqnarray*}%
It follows that%
\begin{eqnarray*}
\frac{(\ref{num2})}{(\ref{den2})} &\leq
&\frac{(\ref{urca})}{\left\{ \sum_{l_{1},l_{2}=0}^{\infty }\Gamma
_{l_{1}}\Gamma
_{l_{2}}(C_{l_{1}0l_{2}0}^{l0})^{2}\right\} ^{2}}\!=\!
\sqrt{\frac{\sum_{l_{1}j_{1}}\frac{\Gamma _{l_{1}}\Gamma _{j_{1}}}{\sqrt[5%
]{j_{1}^{2}}}\left\{ C_{l_{1}0j_{1}0}^{l0}\right\} ^{2}}{%
\sum_{l_{1},l_{2}=0}^{\infty }\Gamma _{l_{1}}\Gamma
_{l_{2}}(C_{l_{1}0l_{2}0}^{l0})^{2}}} \\
&\leq& 2\sqrt{\frac{j^{\ast }\max_{j_{1}\leq j^{\ast }}\left[
\sum_{l_{1}\geq 1}\Gamma _{l_{1}}\Gamma
_{j_{1}}\left\{ C_{l_{1}0j_{1}0}^{l0}\right\} ^{2}\right] }{%
\sum_{l_{1},l_{2}=0}^{\infty }\Gamma _{l_{1}}\Gamma
_{l_{2}}(C_{l_{1}0l_{2}0}^{l0})^{2}}+\frac{1}{\sqrt[5]{(j^{\ast })^{2}}}}%
\text{ .}
\end{eqnarray*}%
Now fix $\varepsilon >0$. Under (\ref{sufcon2}) we have that, for
any fixed and positive number $l_{1}^{\ast }>1/\varepsilon $,
\begin{eqnarray*}
&&\lim_{l\rightarrow \infty }\left[ \frac{j^{\ast }\max_{j_{1}\leq j^{\ast }}%
\left[ \sum_{l_{1}\geq 1}\Gamma _{l_{1}}\Gamma _{j_{1}}\left\{
C_{l_{1}0j_{1}0}^{l0}\right\} ^{2}\right] }{\sum_{l_{1},l_{2}=1}^{\infty
}\Gamma _{l_{1}}\Gamma _{l_{2}}(C_{l_{1}0l_{2}0}^{l0})^{2}}+\frac{1}{\sqrt[5]%
{(j^{\ast })^{2}}}\right] \\
&\leq &j^{\ast }\lim_{l\rightarrow \infty }\sup_{l_{1}}\frac{%
\sum_{l_{2}=1}^{\infty }\Gamma _{l_{1}}\Gamma _{l_{2}}\left\{
C_{l_{1}0l_{2}0}^{l0}\right\} ^{2}}{\sum_{l_{1},l_{2}=1}^{\infty }\Gamma
_{l_{1}}\Gamma _{l_{2}}(C_{l_{1}0l_{2}0}^{l0})^{2}}+\sqrt[5]{\varepsilon ^{2}%
}=\sqrt[5]{\varepsilon ^{2}}\text{ .}
\end{eqnarray*}%
Because $\varepsilon $ is arbitrary, the proof is concluded.
\end{proof}

\textbf{Remark. }Note that, using (\ref{cgconv}) and (\ref{cgconv1}),
condition (\ref{sufcon2}) becomes%
\begin{equation}
\lim_{l\rightarrow \infty }\sup_{\lambda }\frac{\Gamma _{\lambda }\widehat{%
\Gamma }_{2,l;\lambda }^{\ast }}{\sum_{l_{1}}\Gamma _{l_{1}}\widehat{\Gamma }%
_{2,l;l_{1}}}=0\text{ .}  \label{rwhq21}
\end{equation}%
Note also that if, in the convolutions (\ref{cgconv}), one replaces each
squared Clebsch-Gordan coefficient $\left( C_{l_{1}0l_{2}0}^{l0}\right) ^{2}$
by the indicator $\mathbf{1}_{l_{1}+l_{2}=l}$ and extends the sums over $%
\mathbb{Z}$, one obtains the relation
\begin{equation}
\lim_{l\rightarrow \infty }\sup_{l_{1}}\frac{\Gamma _{l_{1}}\Gamma _{l-l_{1}}%
}{\sum_{l_{1}}\Gamma _{l_{1}}\Gamma _{l-l_{1}}}=0\text{.}  \label{rwhq2}
\end{equation}%
In particular, when $\left\{ \Gamma _{l}\right\} =\{ \Gamma
_{l}^{V}\} $ (the power spectrum of the field $V$ on $\mathbb{T}$
given in (\ref{clll})) it is not difficult to show that formula (\ref{rwhq2}%
) gives exactly the asymptotic (necessary and sufficient) condition (\ref%
{rwAbel}).

\subsection{The case $q=3$\label{SS : Q3}}

Our results for $q=3$ closely mirrors the conditions we derived in the
previous subsection.

\begin{proposition}
\label{lemmaq3} A sufficient condition for the CLT (\ref{as00}) when $q=3$
is
\begin{eqnarray}
\lim_{l\rightarrow \infty }\sup_{L_{1}}\frac{\sum_{l_{1}l_{2}j_{1}}\Gamma
_{l_{1}}\Gamma _{l_{2}}\Gamma _{j_{1}}\left\{
C_{l_{1}0l_{2}0j_{1}0}^{L_{1}l0}\right\} ^{2}}{\sum_{L_{1}}%
\sum_{l_{1},l_{2},l_{3}}\Gamma _{l_{1}}\Gamma _{l_{2}}\Gamma
_{l_{3}}\left\{ C_{l_{1}0l_{2}0l_{3}0}^{L_{1}l0}\right\} ^{2}}
&=&0,\text{
and}  \label{condq31} \\
\lim_{l\rightarrow \infty }\sup_{j_{1}}\frac{\sum_{l_{1}l_{2}L_{1}}\Gamma
_{l_{1}}\Gamma _{l_{2}}\Gamma _{j_{1}}\left\{
C_{l_{1}0l_{2}0j_{1}0}^{L_{1}l0}\right\} ^{2}}{\sum_{L_{1}}%
\sum_{l_{1},l_{2},l_{3}}\Gamma _{l_{1}}\Gamma _{l_{2}}\Gamma
_{l_{3}}\left\{ C_{l_{1}0l_{2}0l_{3}0}^{L_{1}l0}\right\} ^{2}}
&=&0\text{ .} \label{condq32}
\end{eqnarray}
\end{proposition}

\textbf{Remark. }In the light of (\ref{cgconv})-(\ref{cgconv1}) and of the
definition of the random walk $Z$ given in (\ref{rw1}) and (\ref{rw2}), it
is not difficult to see that (\ref{condq31}) can be rewritten as%
\begin{align}
\lim_{l\rightarrow \infty }\sup_{\lambda }\frac{\widehat{\Gamma }_{2,\lambda
}\sum_{j_{1}}\Gamma _{j_{1}}\left\{ C_{\lambda j_{1}0}^{l0}\right\} ^{2}}{%
\widehat{\Gamma }_{3,l}}& =\!\lim_{l\rightarrow \infty }\sup_{\lambda }\frac{%
\widehat{\Gamma }_{2,\lambda }\widehat{\Gamma }_{2,l;\lambda }^{\ast }}{%
\sum_{L_{1}}\left[ \widehat{\Gamma }_{2,L_{1}}\widehat{\Gamma }%
_{1,l;L_{1}}^{\ast }\right] }  \label{convq31} \\
& =\!\lim_{l\rightarrow \infty }\sup_{\lambda }\mathbb{P}\left[
Z_{2}=\lambda \mid Z_{3}=l\right] =0\text{.}  \notag
\end{align}%
Likewise, one obtains that (\ref{condq32}) is equivalent to
\begin{eqnarray}
&&\lim_{l\rightarrow \infty }\sup_{j_{1}}\frac{\Gamma _{j_{1}}\widehat{%
\Gamma }_{3,l;j_{1}}^{\ast }}{\sum_{L_{1}}\sum_{l_{1},l_{2},l_{3}}\Gamma
_{l_{1}}\Gamma _{l_{2}}\Gamma _{l_{3}}\left\{
C_{l_{1}0l_{2}0l_{3}0}^{L_{1}l0}\right\} ^{2}}  \label{convq32a} \\
&=&\lim_{l\rightarrow \infty }\sup_{j_{1}}\mathbb{P}\left[ Z_{1}=j_{1}\mid
Z_{3}=l\right] =0\text{ .}  \notag
\end{eqnarray}%
It should be noted that the two conditions (\ref{convq31}) and (\ref%
{convq32a}) can be written compactly as
\begin{equation}
\lim_{l\rightarrow \infty }\max_{q=1,2}\sup_{j_{1}}\frac{\widehat{\Gamma }%
_{q,j_{1}}\widehat{\Gamma }_{3-q,l;j_{1}}^{\ast }}{\sum_{L_{1}}%
\sum_{l_{1},l_{2},l_{3}}\Gamma _{l_{1}}\Gamma _{l_{2}}\Gamma _{l_{3}}\left\{
C_{l_{1}0l_{2}0l_{3}0}^{L_{1}l0}\right\} ^{2}}=0\text{ .}  \label{rwhq3}
\end{equation}%
Relation (\ref{rwhq3}) parallels once again analogous conditions
established for stationary fields on a torus -- see \cite{MaPe}.

\smallskip

\textbf{Proof of Proposition \ref{lemmaq3}. }In view of Part 3 of Theorem %
\ref{teo1}, we shall focus on the asymptotic negligeability of the ratio
appearing in (\ref{as11}), in the case where $q=3$ and $p=2.$ As before, the
denominator of (\ref{as11}) is proportional to%
\begin{eqnarray}
&&\left\{ \sum_{l_{1},l_{2},l_{3}}C_{l_{1}}C_{l_{2}}C_{l_{3}}\frac{1}{2l+1}%
\left\{ \prod_{i=1}^{3}(2l_{i}+1)\right\} \sum_{L_{1}}\left\{
C_{l_{1}0l_{2}0}^{L_{1}0}C_{L_{1}0l_{3}0}^{l0}\right\} ^{2}\right\} ^{2}
\label{den3} \\
&=&\frac{1}{(2l+1)^{2}}\left\{ \sum_{l_{1},l_{2},l_{3}}^{\infty }\Gamma
_{l_{1}}\Gamma _{l_{2}}\Gamma _{l_{3}}\sum_{L_{1}}\left\{
C_{l_{1}0l_{2}0l_{3}0}^{L_{1}l0}\right\} ^{2}\right\} ^{2}.  \notag
\end{eqnarray}%
On the other hand, the numerator is proportional to%
\begin{eqnarray}
&&\frac{1}{(2l+1)^{2}}\sum_{j_{1},j_{2}}\sum_{n_{1},n_{2}}\Gamma
_{j_{1}}\Gamma _{j_{2}}\times  \notag \\
&&\left\vert \sum_{l_{1},l_{2},m_{1},m_{2}}\Gamma _{l_{1}}\Gamma
_{l_{2}}%
\sum_{L_{1}}C_{l_{1}0l_{2}0j_{1}0}^{L_{1}l0}C_{l_{1}m_{1}l_{2}m_{2}j_{1}n_{1}}^{L_{1}lm}\sum_{L_{2}}C_{l_{1}0l_{2}0j_{2}0}^{L_{2}l0}C_{l_{1}m_{1}l_{2}m_{2}j_{2}n_{2}}^{L_{2}lm}\right\vert ^{2}
\notag \\
&=&\frac{1}{(2l+1)^{2}}\sum_{j_{1},j_{2}}\sum_{n_{1},n_{2}}\Gamma
_{j_{1}}\Gamma _{j_{2}}\times  \notag \\
&&\left\vert \sum_{l_{1},l_{2},m_{1},m_{2}}\Gamma _{l_{1}}\Gamma
_{l_{2}}\sum_{L_{1}}C_{l_{1}0l_{2}0j_{1}0}^{L_{1}l0}%
\sum_{M_{1}}C_{l_{1}m_{1}l_{2}m_{2}}^{L_{1}M_{1}}C_{L_{1}M_{1}j_{1}n_{1}}^{lm}\right.
\notag \\
&&\text{ \ \ \ \ \ \ \ \ \ \ \ \ \ \ \ \ \ \ \ \ \ \ \ \ \ \ \ \ \ \ \ \ \ \
\ \ }\left.
\sum_{L_{2}}C_{l_{1}0l_{2}0j_{2}0}^{L_{2}l0}%
\sum_{M_{2}}C_{l_{1}m_{1}l_{2}m_{2}}^{L_{2}M_{2}}C_{L_{2}M_{2}j_{2}n_{2}}^{lm}\right\vert ^{2}
\label{num3}
\end{eqnarray}%
This last expression equals in turn%
\begin{eqnarray}
&=&\frac{1}{(2l+1)^{2}}\sum_{j_{1},j_{2}}\sum_{n_{1},n_{2}}\Gamma
_{j_{1}}\Gamma _{j_{2}}\times  \notag \\
&&\left\vert \sum_{l_{1},l_{2}}\Gamma _{l_{1}}\Gamma
_{l_{2}}\sum_{L_{1}}C_{l_{1}0l_{2}0j_{1}0}^{L_{1}l0}%
\sum_{M_{1}}C_{L_{1}M_{1}j_{1}n_{1}}^{lm}%
\sum_{L_{2}}C_{l_{1}0l_{2}0j_{2}0}^{L_{2}l0}%
\sum_{M_{2}}C_{L_{2}M_{2}j_{2}n_{2}}^{lm}\delta _{L_{1}}^{L_{2}}\delta
_{M_{1}}^{M_{2}}\right\vert ^{2}  \notag \\
&=&\frac{1}{(2l+1)^{2}}\sum_{j_{1},j_{2}}\sum_{n_{1},n_{2}{}_{1}}\Gamma
_{j_{1}}\Gamma _{j_{2}}\times  \notag \\
&&\left\vert \sum_{l_{1},l_{2}=0}\Gamma _{l_{1}}\Gamma
_{l_{2}}%
\sum_{L_{1}}C_{l_{1}0l_{2}0j_{1}0}^{L_{1}l0}C_{l_{1}0l_{2}0j_{2}0}^{L_{1}l0}%
\sum_{M_{1}}C_{L_{1}M_{1}j_{1}n_{1}}^{lm}C_{L_{1}M_{1}j_{2}n_{2}}^{lm}\right%
\vert ^{2}  \notag
\end{eqnarray}%
\begin{eqnarray*}
&=&\frac{1}{(2l+1)^{2}}\sum_{j_{1},j_{2}}\sum_{n_{1},n_{2}{}_{1}}\Gamma
_{j_{1}}\Gamma _{j_{2}}\times \\
&&\left\vert \sum_{l_{1},l_{2}}\Gamma _{l_{1}}\Gamma
_{l_{2}}\sum_{L_{1}}\!C_{l_{1}0l_{2}0j_{1}0}^{L_{1}l0}\!\sum_{M_{1}}%
\!C_{L_{1}M_{1}j_{1}n_{1}}^{lm}\!\sum_{L_{2}}%
\!C_{l_{1}0l_{2}0j_{2}0}^{L_{2}l0}\!\sum_{M_{2}}%
\!C_{L_{2}M_{2}j_{2}n_{2}}^{lm}\delta _{L_{1}}^{L_{2}}\delta
_{M_{1}}^{M_{2}}\!\right\vert ^{2}
\end{eqnarray*}%
and we can use the same argument as for $q=2.$ More precisely, one can write%
\begin{eqnarray}
&&\left\vert \sum_{l_{1},l_{2}=1}\Gamma _{l_{1}}\Gamma
_{l_{2}}%
\sum_{L_{1}}C_{l_{1}0l_{2}0j_{1}0}^{L_{1}l0}C_{l_{1}0l_{2}0j_{2}0}^{L_{1}l0}%
\sum_{M_{1}}C_{L_{1}M_{1}j_{1}n_{1}}^{lm}C_{L_{1}M_{1}j_{2}n_{2}}^{lm}\right%
\vert ^{2}  \notag \\
&=&\sum_{l_{1}...l_{4}}\Gamma _{l_{1}}...\Gamma
_{l_{4}}%
\sum_{L_{1}L_{2}}C_{l_{1}0l_{2}0j_{1}0}^{L_{1}l0}C_{l_{1}0l_{2}0j_{2}0}^{L_{1}l0}C_{l_{3}0l_{4}0j_{1}0}^{L_{2}l0}C_{l_{3}0l_{4}0j_{2}0}^{L_{2}l0}
\notag \\
&&%
\sum_{M_{1}M_{2}}C_{L_{1}M_{1}j_{1}n_{1}}^{lm}C_{L_{1}M_{1}j_{2}n_{2}}^{lm}C_{L_{2}M_{2}j_{1}n_{1}}^{lm}C_{L_{2}M_{2}j_{2}n_{2}}^{lm}
\notag \\
&=&\sum_{l_{1}...l_{4}}\Gamma _{l_{1}}...\Gamma
_{l_{4}}%
\sum_{L_{1}L_{2}}C_{l_{1}0l_{2}0j_{1}0}^{L_{1}l0}C_{l_{1}0l_{2}0j_{2}0}^{L_{1}l0}C_{l_{3}0l_{4}0j_{1}0}^{L_{2}l0}C_{l_{3}0l_{4}0j_{2}0}^{L_{2}l0}
\notag \\
&&(-1)^{\zeta }\sum_{s\sigma }(2s\!+\!1)(2l\!+\!1)(C_{l0s\sigma
}^{l0})^{2}\!\left\{
\begin{array}{ccc}
L_{1} & j_{1} & l \\
l & s & L_{2}%
\end{array}%
\right\} \!\left\{
\begin{array}{ccc}
L_{1} & j_{2} & l \\
l & s & L_{2}%
\end{array}%
\right\}  \label{so}
\end{eqnarray}%
where $\zeta =L_{1}+j_{1}+L_{2}+j_{2}$, and (\ref{so}) equals%
\begin{eqnarray*}
&=&\sum_{l_{1}...l_{4}}\Gamma _{l_{1}}...\Gamma
_{l_{4}}%
\sum_{L_{1}L_{2}}C_{l_{1}0l_{2}0j_{1}0}^{L_{1}l0}C_{l_{1}0l_{2}0j_{2}0}^{L_{1}l0}C_{l_{3}0l_{4}0j_{1}0}^{L_{2}l0}C_{l_{3}0l_{4}0j_{2}0}^{L_{2}l0}
\\
&&(-1)^{2l}\!\sum_{s}\!(2s\!+\!1)(2l\!+\!1)(C_{lms0}^{lm})^{2}\!\left\{
\begin{array}{ccc}
L_{1} & j_{1} & l \\
l & s & L_{2}%
\end{array}%
\right\} \!\left\{
\begin{array}{ccc}
L_{1} & j_{2} & l \\
l & s & L_{2}%
\end{array}%
\right\} .
\end{eqnarray*}%
From (\ref{lemma6j}) we now obtain that the last expression is bounded by%
\begin{eqnarray*}
&&\sum_{l_{1}...l_{4}}\Gamma _{l_{1}}...\Gamma
_{l_{4}}%
\sum_{L_{1}L_{2}}C_{l_{1}0l_{2}0j_{1}0}^{L_{1}l0}C_{l_{1}0l_{2}0j_{2}0}^{L_{1}l0}C_{l_{3}0l_{4}0j_{1}0}^{L_{2}l0}C_{l_{3}0l_{4}0j_{2}0}^{L_{2}l0}%
\frac{1}{\sqrt[5]{L_{1}}} \\
&&+\sum_{l_{1}...l_{4}}\Gamma _{l_{1}}...\Gamma
_{l_{4}}%
\sum_{L_{1}L_{2}}C_{l_{1}0l_{2}0j_{1}0}^{L_{1}l0}C_{l_{1}0l_{2}0j_{2}0}^{L_{1}l0}C_{l_{3}0l_{4}0j_{1}0}^{L_{2}l0}C_{l_{3}0l_{4}0j_{2}0}^{L_{2}l0}%
\frac{1}{\sqrt[5]{j_{1}}}\text{ ,}
\end{eqnarray*}%
whence all the terms are bounded by
\begin{eqnarray}
&&\sum_{j_{1}j_{2}}\sum_{l_{1}...l_{4}}\sum_{L_{1}L_{2}}\Gamma
_{j_{1}}\Gamma _{j_{2}}\Gamma _{l_{1}}...\Gamma
_{l_{4}}C_{l_{1}0l_{2}0}^{L_{1}0}C_{L_{1}0j_{1}0}^{l0}C_{l_{1}0l_{2}0}^{L_{1}0}\times
\label{q3bou1} \\
&&\text{ \ \ \ \ \ \ \ \ \ \ \ \ \ \ \ \ \ \ \ \ \ \ \ \ \ \ \ \ \ \ \ \ \ }%
\times
C_{L_{1}0j_{2}0}^{l0}C_{l_{3}0l_{4}0}^{L_{2}0}C_{L_{2}0j_{1}0}^{l0}C_{l_{3}0l_{4}0}^{L_{2}0}C_{L_{2}0j_{2}0}^{l0}%
\frac{1}{\sqrt[5]{L_{1}}}  \notag \\
&&+\sum_{j_{1}j_{2}}\sum_{l_{1}...l_{4}}\sum_{L_{1}L_{2}}\Gamma
_{j_{1}}\Gamma _{j_{2}}\Gamma _{l_{1}}...\Gamma
_{l_{4}}C_{l_{1}0l_{2}0}^{L_{1}0}C_{L_{1}0j_{1}0}^{l0}C_{l_{1}0l_{2}0}^{L_{1}0}\times
\label{q3bou2} \\
&&\text{ \ \ \ \ \ \ \ \ \ \ \ \ \ \ \ \ \ \ \ \ \ \ \ \ \ \ \ \ \ \ \ \ }%
\times
C_{L_{1}0j_{2}0}^{l0}C_{l_{3}0l_{4}0}^{L_{2}0}C_{L_{2}0j_{1}0}^{l0}C_{l_{3}0l_{4}0}^{L_{2}0}C_{L_{2}0j_{2}0}^{l0}%
\frac{1}{\sqrt[5]{j_{1}}}\text{ .}  \notag
\end{eqnarray}%
Also,
\begin{eqnarray*}
&&\sum_{j_{1}j_{2}}\sum_{l_{1}...l_{4}}\Gamma _{j_{1}}\Gamma _{j_{2}}\Gamma
_{l_{1}}...\Gamma _{l_{4}}\sum_{L_{1}L_{2}}\!\frac{%
C_{l_{1}0l_{2}0j_{1}0}^{L_{1}l0}C_{l_{1}0l_{2}0j_{2}0}^{L_{1}l0}C_{l_{3}0l_{4}0j_{1}0}^{L_{2}l0}C_{l_{3}0l_{4}0j_{2}0}^{L_{2}l0}%
}{\sqrt[5]{L_{1}}} \\
&=&\sum_{\substack{ j_{1}j_{2}  \\ l_{1}...l_{4}}}\Gamma _{j_{1}}\Gamma
_{j_{2}}\Gamma _{l_{1}}...\Gamma _{l_{4}}\sum_{L_{1}...L_{4}}\!\frac{%
C_{l_{1}0l_{2}0j_{1}0}^{L_{1}l0}C_{l_{1}0l_{2}0j_{2}0}^{L_{3}l0}C_{l_{3}0l_{4}0j_{1}0}^{L_{2}l0}C_{l_{3}0l_{4}0j_{2}0}^{L_{4}l0}%
}{\sqrt[5]{L_{1}}}\delta _{L_{1}}^{L_{3}}\delta _{L_{2}}^{L_{4}} \\
&\leq &\sqrt{\sum_{l_{1}l_{2}j_{1}L_{1}}\Gamma _{l_{1}}\Gamma _{l_{2}}\Gamma
_{j_{1}}\frac{\left\{ C_{l_{1}0l_{2}0j_{1}0}^{L_{1}l0}\right\} ^{2}}{%
L_{1}^{2/5}}}\sqrt{\sum_{l_{1}l_{2}j_{2}L_{1}}\Gamma _{l_{1}}\Gamma
_{l_{2}}\Gamma _{j_{2}}\left\{ C_{l_{1}0l_{2}0j_{2}0}^{L_{1}l0}\right\} ^{2}}%
\times \\
&&\sqrt{\sum_{l_{3}l_{4}j_{1}L_{2}}\Gamma _{l_{3}}\Gamma _{l_{4}}\Gamma
_{j_{1}}\left\{ C_{l_{3}0l_{4}0j_{1}0}^{L_{2}l0}\right\} ^{2}}\sqrt{%
\sum_{l_{3}l_{4}j_{2}L_{2}}\Gamma _{l_{3}}\Gamma _{l_{4}}\Gamma
_{j_{2}}\left\{ C_{l_{3}0l_{4}0j_{2}0}^{L_{2}l0}\right\} ^{2}}
\end{eqnarray*}\\[-30pt]
\begin{equation}
\leq \sqrt{\sum_{l_{1}l_{2}j_{1}L_{1}}\!\Gamma _{l_{1}}\Gamma _{l_{2}}\Gamma
_{j_{1}}\!\frac{\left\{ C_{l_{1}0l_{2}0j_{1}0}^{L_{1}l0}\right\} ^{2}}{%
L_{1}^{2/5}}}\!\left\{ \sum_{l_{1}l_{2}j_{2}L_{1}}\!\Gamma _{l_{1}}\Gamma
_{l_{2}}\Gamma _{j_{2}}\!\left\{ C_{l_{1}0l_{2}0j_{2}0}^{L_{1}l0}\right\}
^{2}\!\right\} ^{3/2}\text{.}  \label{socc}
\end{equation}%
To sum up, we have obtained%
\begin{eqnarray}
\frac{(\ref{q3bou1})}{(\ref{den3})} &\leq &\frac{\frac{1}{(2l+1)^{2}}%
\times (\ref{socc})}{\frac{1}{ (2l+1)^{2}}\left\{
6\sum_{l_{1},l_{2},l_{3}}\Gamma _{l_{1}}\Gamma _{l_{2}}\Gamma
_{l_{3}}\sum_{L_{1}}\left\{
C_{l_{1}0l_{2}0l_{3}0}^{L_{1}l0}\right\}
^{2}\right\} ^{2}} \notag \\
&\leq &\left[ \frac{\sum_{l_{1}l_{2}j_{1}L_{1}}\Gamma
_{l_{1}}\Gamma _{l_{2}}\Gamma _{j_{1}}\left\{
C_{l_{1}0l_{2}0j_{1}0}^{L_{1}l0}\right\}
^{2}/L_{1}^{2/5}}{6\sum_{l_{1},l_{2},l_{3}}\Gamma _{l_{1}}\Gamma
_{l_{2}}\Gamma _{l_{3}}\sum_{L_{1}}\left\{
C_{l_{1}0l_{2}0l_{3}0}^{L_{1}l0}\right\} ^{2}}\right] ^{1/2}
\label{joe},
\end{eqnarray}%
By an identical argument we obtain also
\begin{equation}
\frac{(\ref{q3bou2})}{(\ref{den3})}\leq \left[ \frac{%
\sum_{l_{1}l_{2}j_{1}L_{1}}\Gamma _{l_{1}}\Gamma _{l_{2}}\Gamma
_{j_{1}}\left\{ C_{l_{1}0l_{2}0j_{1}0}^{L_{1}l0}\right\} ^{2}/j_{1}^{2/5}}{%
6\sum_{l_{1},l_{2},l_{3}}\Gamma _{l_{1}}\Gamma _{l_{2}}\Gamma
_{l_{3}}\sum_{L_{1}}\left\{ C_{l_{1}0l_{2}0l_{3}0}^{L_{1}l0}\right\} ^{2}}%
\right] ^{1/2}.  \label{zzzz}
\end{equation}%
Now we can adopt exactly the same line of reasoning as in the proof of
Proposition \ref{lemmaq2}, so that by trivial manipulations we deduce that (%
\ref{condq31}) and (\ref{condq32}) are indeed sufficient to have
that the RHS of (\ref{joe}) and (\ref{zzzz}) converges to zero as
$l\rightarrow +\infty $. \hfill $\square $

\subsection{The case of a general $q$: results and conjectures\label{SS :
CONJQ}}

The following proposition gives a general version of the results proved in
Sections \ref{SS : Q2} and \ref{SS : Q3}. The proof (omitted) is rather
long, and can be obtained along the lines of those of Proposition \ref%
{lemmaq2}\ and Proposition \ref{lemmaq3}.

\begin{proposition}
\label{P : qq-1}Fix $q\geq 4$. Then, a sufficient condition to have the
asymptotic relation (\ref{as11}) in the case $p=q-1$ is the following;%
\begin{eqnarray}
&&\lim_{l\rightarrow \infty }\left\{ \sup_{\lambda }\frac{\widehat{\Gamma }%
_{q-1,\lambda }\widehat{\Gamma }_{2,l;\lambda }^{\ast }}{\sum_{L}\widehat{%
\Gamma }_{q-1,L}\widehat{\Gamma }_{1,l;L}^{\ast }}+\sup_{\lambda }\frac{%
\widehat{\Gamma }_{q,l;\lambda }^{\ast }\Gamma _{\lambda }}{\sum_{L}\widehat{%
\Gamma }_{q-1,L}\widehat{\Gamma }_{1,l;L}^{\ast }}\right\}  \notag \\
&=&\lim_{l\rightarrow \infty }\left\{ \sup_{\lambda }\frac{\widehat{\Gamma }%
_{q-1,\lambda }\widehat{\Gamma }_{2,l;\lambda }^{\ast }}{\widehat{\Gamma }%
_{q,l}}+\sup_{\lambda }\frac{\widehat{\Gamma }_{q,l;\lambda }^{\ast }\Gamma
_{\lambda }}{\widehat{\Gamma }_{q,l}}\right\} =0\text{.}  \label{qq-1}
\end{eqnarray}
\end{proposition}

\textbf{Remarks. }(\textbf{1}) As in the proofs of Proposition \ref{lemmaq2}%
\ and Proposition \ref{lemmaq3}, a crucial technique in proving Proposition %
\ref{P : qq-1} consists in the simplification of sums of the type%
\begin{equation}
\sum_{\substack{ m_{1}m_{2}m_{3}  \\ M_{1}...M_{4}}}%
C_{l_{1}m_{1}l_{2}m_{2}}^{L_{1}M_{1}}C_{L_{1}M_{1}l_{3}m_{3}}^{L_{2}M_{2}}C_{L_{2}M_{2}j_{1}n_{1}}^{lm}C_{l_{1}m_{1}l_{2}m_{2}}^{L_{3}M_{3}}C_{L_{3}M_{3}l_{3}m_{3}}^{L_{4}M_{4}}C_{L_{4}M_{4}j_{2}n_{2}}^{lm2}%
\text{,}  \label{goo}
\end{equation}%
by means of the general relation%
\begin{equation}
\sum_{m_{1}m_{2}}C_{l_{1}m_{1}l_{2}m_{2}}^{L_{1}M_{1}}C_{l_{1}m_{1}l_{2}m_{2}}^{L_{3}M_{3}}=\delta _{L_{1}}^{L_{3}}\delta _{M_{1}}^{M_{3}}.
\label{fff}
\end{equation}%
This basically means that, if in (\ref{goo}) each Clebsch-Gordan coefficient
is represented as the vertex of a connected graph, then it is possible to
\textquotedblleft reduce\textquotedblright\ such graph by cutting edges
corresponding to 2-loops -- see \cite{Marinucci} for a more detailed
discussion on these graphical methods.

(\textbf{2}) Note that, since $q\geq 4$ and according to Part C of Theorem %
\ref{teo1}, condition (\ref{as11}) \textsl{is only necessary} to have the
CLT (\ref{as00}), so that (\ref{qq-1}) cannot be used to deduce the
asymptotic Gaussianity of the frequency components of Hermite-subordinated
fields of the type $H_{q}\left[ T\right] $. Some conjectures concerning the
case $q\geq 4$, $p\neq q-1$ are presented at the end of the section.

(\textbf{3}) Observe that, in terms of the random walk $\left\{
Z_{n}\right\} $ defined in (\ref{rw1})-(\ref{rw2}),%
\begin{equation*}
\frac{\widehat{\Gamma }_{q-1,\lambda }\widehat{\Gamma
}_{2,l;\lambda }^{\ast }}{\widehat{\Gamma
}_{q,l}}=\mathbb{P}\left\{ Z_{q-1}=\lambda \mid Z_{q}=l\right\}
\text{  \ and \  }\frac{\widehat{\Gamma }_{q,l;\lambda }^{\ast
}\Gamma _{\lambda }}{\widehat{\Gamma }_{q,l}}=\mathbb{P}\left\{
Z_{1}=\lambda \mid Z_{q}=l\right\} .
\end{equation*}

\smallskip

As mentioned before, the relation (\ref{as11}) (which implies (\ref{as00})),
in the general case where $q\geq 4$ and $p\neq q-1$, is still being
investigated, as it requires a hard analysis of higher order Clebsch-Gordan
coefficients by means of graphical techniques (see for instance \cite[Ch. 11]%
{VMK}). At this stage, it is however natural to propose the following
conjecture. Recall that we focus on the CLT (\ref{as00}) because of the
equality in law $T_{l}^{\left( q\right) }\left( x\right) =\sqrt{\frac{2l+1}{%
4\pi }}a_{l0;q}$, and Corollary \ref{C : PunctualCLT}.


\textbf{Conjecture A }(\textit{Weak})\textbf{\ }\textsl{A sufficient
condition for the CLT (\ref{as00}) is}%
\begin{eqnarray}
&&\lim_{l\rightarrow \infty }\max_{1\leq p\leq q-1}\sup_{\lambda }\frac{%
\widehat{\Gamma }_{p,\lambda }\widehat{\Gamma }_{q+1-p,l;\lambda }^{\ast }}{%
\sum_{L}\widehat{\Gamma }_{p,L_{q-2}}\widehat{\Gamma }_{q+1-p,l;L}^{\ast }}
\label{rwhqw} \\
&=&\lim_{l\rightarrow \infty }\max_{1\leq p\leq q-1}\sup_{\lambda }\mathbb{P}%
\left\{ Z_{p}=\lambda \mid Z_{q}=l\right\} =0\text{ .}  \notag
\end{eqnarray}


It is worth emphasizing how condition (\ref{rwhqw}) is the exact analogous
of the necessary and sufficient condition (\ref{rwAbel}), established in
\cite{MaPe} for the high-frequency CLT\ on the torus $\mathbb{T}=[0,2\pi )$.
This remarkable circumstance may suggest the following (much more general
and, for the time being, quite imprecise) extension.


\textbf{Conjecture B }(\textit{Strong})\textbf{\ }\textsl{Let }$T$ \textsl{%
be an isotropic Gaussian field defined on the homogeneous space of a}
\textsl{compact group }$G$, \textsl{and set }$T^{\left( q\right)
}=H_{q}\left( T\right) $ ($q\geq 2$). \textsl{Then, the high-frequency}
\textsl{components of }$T^{\left( q\right) }$\textsl{\ are asymptotically
Gaussian if, and only if, it holds a condition of the type}%
\begin{equation}
\lim_{l\rightarrow l_{0}}\max_{1\leq p\leq q-1}\sup_{\lambda \in \widehat{G}}%
\frac{\widehat{\Gamma }_{p,\lambda }^{\ast }\widehat{\Gamma }%
_{q+1-p,l;\lambda }}{\sum_{L\in \widehat{G}}\widehat{\Gamma }_{p,L}^{\ast }%
\widehat{\Gamma }_{q+1-p,l;L}}=0\text{ ,}  \label{rwhqs}
\end{equation}%
\textsl{where }$\widehat{G}$ \textsl{is the dual of} $G,$ $l_{0}$ \textsl{is
some point at the boundary of} $\widehat{G}$\textsl{, and the convolutions }$%
\widehat{\Gamma }$\textsl{\ and }$\widehat{\Gamma }^{\ast }$\textsl{\ are
defined (analogously to (\ref{cgconv-2})-(\ref{starconv})) on the power
spectrum of }$T$,\textsl{\ by means of the appropriate Clebsch-Gordan
coefficients of the group.}


We leave the two Conjectures A and B as open issues for future research.


\textbf{Remark. }(\textit{On} \textquotedblleft \textit{no privileged path}%
\textquotedblright \textit{\ conditions}) In terms of $Z$, condition (\ref%
{rwhqw}) can be further interpreted as follows: for every $l$, define a
\textquotedblleft bridge\textquotedblright\ of length $q$, by conditioning $%
Z $ to equal $l$ at time $q$. Then, (\ref{rwhqw}) is verified if, and only
if, the probability that the bridge hits $\lambda $ at time $q$ converges to
zero, uniformly on $\lambda $, as $l\rightarrow +\infty $. It is also
evident that, when (\ref{rwhqw}) is verified for every $p=1,...,q-1$, one
also has that
\begin{equation}
\lim_{l\rightarrow +\infty }\sup_{\lambda _{1},...,\lambda _{q-1}\in \mathbb{%
N}}\mathbb{P}\left[ Z_{1}=\lambda _{1},...,Z_{m-1}=\lambda _{q-1}\mid Z_{q}=l%
\right] =0\text{,}  \label{rwint}
\end{equation}%
meaning that, asymptotically, the law of $Z$ does not charge any
\textquotedblleft privileged path\textquotedblright\ of length $q$\ leading
to $l$. The interpretation of condition (\ref{rwint}) in terms of bridges
can be reinforced by putting by convention $Z_{0}=0$, so that the
probability in (\ref{rwint}) is that of the particular path $0\rightarrow
\lambda _{1}\rightarrow ...\rightarrow \lambda _{q-1}\rightarrow l$,
associated with a random bridge linking $0$ and $l$.

\section{Further physical interpretation of the convolutions and connection
with other random walks on hypergroups\label{S : Roynette}}

\subsection{Convolutions as mixed states}

We recall that, in quantum mechanics, it is customary to consider two
possible initial states for a particle, i.e. those provided by the so-called
\textsl{pure states},\emph{\ }where the state of a particle is given, and
those provided by the so-called \textsl{mixed states},\emph{\ }where the
state of the particle is given by a mixture (in the usual probabilistic
sense) over different quantum states. We refer the reader to \cite{Libo} for
an introduction to these ideas. From this standpoint, the quantity $\widehat{%
\Gamma }_{q,l}$ defined in (\ref{cgconv}) is the probability associated to a
mixed state, where the mixing is performed over all possible values of the
total angular momentum. To illustrate this point, we use the standard
bra-ket notation $\left| l0\right\rangle $ to indicate the state of a
particle having total angular momentum equal to $l$ and projection $0$ on
the $z$-axis. By using this formalism, the quantity $\widehat{\Gamma }_{q,l}$
can be obtained as follows:

\begin{description}
\item[(i)] consider a system of $q$ particles $\alpha _{1},...,\alpha _{q}$
such that each $\alpha _{j}$ is in the mixed state $\Xi $ according to which
a particle is in the state $\left\vert k0\right\rangle $ with probability $%
\Gamma _{k}/\Gamma _{\ast }$ ($k\geq 0$);

\item[(ii)] obtain $\widehat{\Gamma }_{q,l}$ as the probability that the
elements of this system are coupled pairwise to form a particle in the state
$\left\vert l0\right\rangle $.
\end{description}

Now denote by $\mathbf{A}_{p,\left\vert \lambda 0\right\rangle }$
the event that the first $p$ particles $\alpha _{1},...,\alpha
_{p}$ have coupled pairwise to generate the state $\left\vert
\lambda
0\right\rangle $. Then, one also has that%
\begin{equation}
\frac{\widehat{\Gamma }_{p+1,\lambda }\widehat{\Gamma }_{q-p,l;\lambda
}^{\ast }}{\widehat{\Gamma }_{q,l}}=\Pr \text{ }\{\text{the }q\text{
particles generate }\left\vert l0\right\rangle \text{ \ }|\text{ \ }\mathbf{A%
}_{p,\left\vert \lambda 0\right\rangle }\text{ }\}.  \label{ourpint}
\end{equation}%
In particular, relation (\ref{ourpint}) yields a further physical
interpretation of the \textquotedblleft no privileged path
condition\textquotedblright\ discussed in (\ref{rwint}).

\subsection{Other convolutions and random walks on group duals}

Random walks on hypergroups, and specifically on group duals, have been
actively studied in the seventies -- see \cite[Ch. 6]{GKR}. Our aim in the
sequel is to compare our definitions with those provided in this earlier
literature, mainly by discussing the alternative physical meanings of the
associated notion of convolution. We recall from Section \ref{S :
Clebsch-Gordan} that, starting from the Wigner's $D$-matrices representation
of $SO(3)$, we obtain the unitary equivalent reducible representations $%
\{ D^{l_{1}}(g)\otimes D^{l_{2}}(g)\} $ and $\{ \oplus
_{l=|l_{2}-l_{1}|}^{l_{2}+l_{1}}D^{l}(g)\} $. Now note $\chi
_{l}(g)$
the character of $D^{l}(g)$; for all $g\in SO(3)$, we have immediately%
\begin{equation*}
\chi _{l_{1}}(g)\chi _{l_{2}}(g)=\sum_{l=|l_{2}-l_{1}|}^{l_{2}+l_{1}}\chi
_{l}(g)\text{ .}
\end{equation*}%
In \cite[p. 222]{GKR}, an alternative class of Clebsch-Gordan coefficients $%
\{C_{l_{1}l_{2}|G}^{l}$ $:$ $l_{1},l_{2},l\geq 0\}$ is defined by means of
the identity%
\begin{equation*}
\frac{1}{2l_{1}+1}\chi _{l_{1}}(g)\frac{1}{2l_{2}+1}\chi
_{l_{2}}(g)=\sum_{l}C_{l_{1}l_{2}|G}^{l}\frac{1}{2l+1}\chi _{l}(g)
\end{equation*}%
which leads to%
\begin{equation*}
C_{l_{1}l_{2}|G}^{l}=\frac{2l+1}{(2l_{1}+1)(2l_{2}+1)}\left\{
l_{1}l_{2}l\right\} \text{,}
\end{equation*}%
where we use the same notation as in \cite{VMK} and in many other physical
textbooks, i.e. we take $\left\{ l_{1}l_{2}l\right\} $ to represent the
indicator function of the event $|l_{2}-l_{1}|\leq l\leq l_{2}+l_{1}$. Of
course%
\begin{equation}
C_{l_{1}l_{2}|G}^{l}=\sum_{l=|l_{2}-l_{1}|}^{l_{2}+l_{1}}\frac{2l+1}{%
(2l_{1}+1)(2l_{2}+1)}\equiv 1\text{ .}  \label{sum}
\end{equation}%
As observed in \cite{GKR}, relation (\ref{sum}) can be used to endow $%
\widehat{SO\left( 3\right) }$ with an hypergroup structure, via the formal
addition $l_{1}+l_{2}\triangleq \sum_{l}lC_{l_{1}l_{2}|G}^{l}$. Now let $%
\left\{ \Gamma _{l}:l\geq 0\right\} $ be a collection of positive
coefficients such that $\sum_{l}\Gamma _{l}=1$. The convolutions and
*-convolutions of the $\left\{ \Gamma _{l}\right\} $ that are naturally
associated with the above formal addition are given by%
\begin{eqnarray}
\widetilde{\Gamma }_{2,l} &=&\sum_{l_{1},l_{2}}\Gamma _{l_{1}}\Gamma
_{l_{2}}C_{l_{1}l_{2}|G}^{l}\text{ , \ }\widetilde{\Gamma }%
_{3,l}=\sum_{L_{1},l_{3}}\widetilde{\Gamma }_{2,L_{1}}\Gamma
_{l_{3}}C_{L_{1}l_{3}|G}^{l}\text{, ...}  \label{Roy1} \\
\widetilde{\Gamma }_{q,l} &=&\sum_{L_{1},l_{q}}\widetilde{\Gamma }%
_{q-1,L_{q-1}}\Gamma _{l_{q}}C_{L_{q-1}l_{q}|G}^{l}\text{ },  \label{Roy3}
\end{eqnarray}%
and, for $p\geq 2$,%
\begin{equation}
\widetilde{\Gamma }_{p,l;l_{1}}^{\ast }=\sum_{l_{2}}\cdot \cdot \cdot
\sum_{l_{p}}\Gamma _{l_{2}}\cdot \cdot \cdot \Gamma
_{l_{p}}%
\sum_{L_{1}...L_{p-2}}C_{l_{1}l_{2}|G}^{L_{1}}C_{L_{1}l_{3}|G}^{L_{2}}...C_{L_{p-2}l_{p}|G}^{l}%
\text{ }.  \label{Roy4}
\end{equation}

As shown in \cite{GKR}, the objects appearing in
(\ref{Roy1})-(\ref{Roy4}) can be used to define the law of a
random walk $\widetilde{Z}=\{ \widetilde{Z}_{n}:n\geq 1\} $ on
$\mathbb{N}$ (regarded as an hypergroup isomorphic to
$\widehat{SO\left( 3\right) }$), exactly as we did in
(\ref{rw1})-(\ref{rw2}). In particular, since $\Gamma _{\ast
}=\sum
\Gamma _{l}=1$, one has that $\widetilde{\Gamma }_{p,l:l_{1}}^{\ast }=%
\mathbb{P}\{ \widetilde{Z}_{p}=l\mid \widetilde{Z}_{1}=l_{1}\} $.
Also, the convolutions (\ref{Roy1})-(\ref{Roy4}) (and therefore
the random walk $\widetilde{Z}$) enjoy a physical interpretation
which is interesting to compare with our previous result. To see
this, assume we have two mixed states $\Xi _{l_{1}}$ and$\ \Xi
_{l_{2}}$: in state $\Xi _{l_{1}}$, the
particle has total angular momentum $l_{1}$ and its projection on the axis $%
z $ takes values $m_{1}=-l_{1},...,l_{1}$ with uniform (classical)
probability $(2l_{1}+1)^{-1}$; analogous conditions are imposed for $\Xi
_{l_{2}}$. Let us now compute the probability $\Pr \left\{ l\mid \Xi
_{l_{1}},\Xi _{l_{2}}\right\} $ that the system will couple to form a
particle with total angular momentum $l$ and arbitrary projection on $z$.
Start by observing that the probability that a particle in the state $%
\left\vert l_{1}m_{1}\right\rangle $ will couple with another particle in
the state $\left\vert l_{2}m_{2}\right\rangle $ to yield the state $%
\left\vert lm\right\rangle $ is exactly given by $%
\{C_{l_{1}m_{1}l_{2}m_{2}}^{lm}\}^{2}. $ Hence, with straightforward
notation,%
\begin{eqnarray}
\Pr \left\{ l\mid \Xi _{l_{1}},\Xi _{l_{2}}\right\} &=&\sum_{m_{1}m_{2}}\Pr
\left\{ l\mid \left\vert l_{1}m_{1}\right\rangle ,\left\vert
l_{2}m_{2}\right\rangle \right\} \Pr \left\{ m_{1},m_{2}\right\}  \notag \\
&=&\sum_{m_{1}m_{2}}\Pr \left\{ l\mid \left\vert l_{1}m_{1}\right\rangle
,\left\vert l_{2}m_{2}\right\rangle \right\} \frac{1}{2l_{1}+1}\frac{1}{%
2l_{2}+1}  \notag \\
&=&\sum_{m}\sum_{m_{1}m_{2}}\left\{ C_{l_{1}m_{1}l_{2}m_{2}}^{lm}\right\}
^{2}\frac{1}{2l_{1}+1}\frac{1}{2l_{2}+1}  \notag \\
&=&\sum_{m}\frac{\left\{
l_{1}l_{2}l\right\}}{2l_{1}+1}\frac{1}{2l_{2}+1}=\frac{2l+1}{2l_{1}+1}\frac{\left\{
l_{1}l_{2}l\right\}}{2l_{2}+1}=C_{l_{1}l_{2}|G}^{l}\text{ }.
\label{cgr}
\end{eqnarray}%
It follows from (\ref{cgr}) that the quantity $\widetilde{\Gamma }_{q,l}$
can be obtained as follows:

\begin{description}
\item[(i)] consider a system of $q$ particles $\alpha _{1},...,\alpha _{q}$
such that each $\alpha _{j}$ is in the mixed state $\Xi $ according to which
a particle is in the state $\left\vert ku\right\rangle $, $u=-k,...,k$, with
probability $(2k+1)^{-1}\Gamma _{k}/\Gamma _{\ast }$ ($k\geq 0$);

\item[(ii)] obtain $\widetilde{\Gamma }_{q,l}$ as the probability that the
elements of this system are coupled pairwise to form a particle in the state
$\left| lm\right\rangle ,$ any $m=-l,...,l.$
\end{description}

To sum up, both convolutions $\widehat{\Gamma }$ and $\widetilde{\Gamma }$
can be interpreted in terms of random interacting quantum particles: $%
\widehat{\Gamma }$-type convolutions are obtained from particles in mixed
states where the mixing is performed over pure states of the form $%
\left\vert k0\right\rangle $; on the other hand, $\widetilde{\Gamma }$-type
convolutions are associated with mixed state particles where mixing is over
pure states of the type $\left\{ \left\vert ku\right\rangle
:u=-k,...,k\right\} $, uniformly in $u$ for every fixed $k.$

\section{Application: algebraic/exponential dualities\label{S : Ang PS}}

In this section we discuss explicit conditions on the angular
power spectrum $\left\{ C_{l}:l\geq 0\right\} $ of the Gaussian
field $T$ introduced in Section \ref{S : GaussSub}, ensuring that
the CLT (\ref{as00}) may hold. Our results show that, if the power
spectrum decreases exponentially, then a high-frequency CLT\
holds, whereas the opposite implication holds if the spectrum
decreases as a negative power. This duality mirrors analogous
conditions previously established in the Abelian case -- see
\cite{MaPe}. For simplicity, we stick to the case $q=2.$ Note that
the results below allow to deal with the asymptotic
(high-frequency) behaviour of the Sachs-Wolfe model
(\ref{swmodel}).

\subsection{The Exponential case}

Assume
\begin{equation}
C_{l}\approx (l+1)^{\alpha }\exp (-l)\text{ },\text{ }\alpha \geq 0.
\label{fdsq}
\end{equation}%
To prove that, in this case, (\ref{as00}) is verified for $q=2$, we will
prove that (\ref{sufcon2}) holds (recall the definition of $\Gamma _{l}$
given in (\ref{FreqSpectrum})). For the denominator of the previous
expression we obtain the lower bound%
\begin{eqnarray}
\sum_{l_{1},l_{2}=1}^{\infty }\Gamma _{l_{1}}\Gamma
_{l_{2}}(C_{l_{1}0l_{2}0}^{l0})^{2} &\geq &\sum_{l_{1}=[l/3]}^{[2l/3]}\Gamma
_{l_{1}}\Gamma _{l-l_{1}}(C_{l_{1}0l-l_{1}0}^{l0})^{2}  \notag \\
&\approx &\exp (-l)l^{2(\alpha
+1)}\sum_{l_{1}=[l/3]}^{[2l/3]}(C_{l_{1}0l-l_{1}0}^{l0})^{2}  \label{seg2}
\end{eqnarray}%
and in view of \cite{VMK}, equation 8.5.2.33, and Stirling's formula%
\begin{eqnarray*}
(\ref{seg2}) &\approx &\exp (-l)l^{2(\alpha
+1)}\sum_{l_{1}=[l/3]}^{[2l/3]}\left( \frac{l!}{l_{1}!(l-l_{1})!}\right)
^{2}\left( \frac{(2l_{1})!(2l-2l_{1})!}{(2l)!}\right) \\
&\approx &\exp (-l)l^{2(\alpha +1)}\sum_{l_{1}=[l/3]}^{[2l/3]}\frac{l^{2l+1}%
}{l_{1}^{2l_{1}+1}(l-l_{1})^{2l-2l_{1}+1}} \\
&&\times \left( \frac{(2l_{1})^{2l_{1}+1/2}(2l-2l_{1})^{2l-2l_{1}+1/2}}{%
(2l)^{2l+1/2}}\right) \\
&\approx &\exp (-l)l^{2(\alpha +1)}\sum_{l_{1}=[l/3]}^{[2l/3]}\frac{l^{1/2}}{%
l_{1}^{1/2}(l-l_{1})^{1/2}}\approx \exp (-l)l^{2(\alpha +1)}l^{1/2}.
\end{eqnarray*}%
On the other hand, recall that by the triangle conditions (Section
\ref{S : Clebsch-Gordan}) $\{ C_{l_{1}0l_{2}0}^{l0}\} ^{2}$
$\equiv 0$ unless
$l_{1}+l_{2}\geq l.$ Hence%
\begin{eqnarray*}
&&\sup_{l_{1}}\sum_{l_{2}}\Gamma _{l_{1}}\Gamma _{l_{2}}\left\{
C_{l_{1}0l_{2}0}^{l0}\right\}^{2} \leq \!K\!\sup_{l_{1}}\exp
(-l)l_{1}^{\alpha +1}\!\\
&&\times \left\{ \left\vert l\!-\!l_{1}\right\vert ^{\alpha +1} \!
+ \!\sum_{u=1}^{\infty }\!\exp (-u)\left\vert
l_{1}\!+\!u\right\vert ^{\alpha +1}\!\right\} \! \approx \!
\exp(-l)l^{2(\alpha +1)}l^{1/2}.
\end{eqnarray*}%
It is then immediate to see that that (\ref{sufcon2}) is satisfied.

\subsection{Regularly varying functions}

For $q=2$, we show below that the CLT fails for all sequences $C_{l}$ such
that: (a) $C_{l}$ is quasi monotonic, i.e. $C_{l+1}\leq C_{l}(1+K/l)$, and
(b) $C_{l}$ is such that $\lim \inf_{l\rightarrow \infty }C_{l}/C_{l/2}>0$.
In particular, a necessary condition for the CLT (\ref{as00}) to hold is that%
\emph{\ }$C_{l}/C_{l/2}\rightarrow 0$. This is exactly the same
necessary condition as was derived by \cite{MaPe} in the Abelian
case. For the general case $q\geq 2$, we expect the CLT fails for
all regularly varying angular power spectra, i.e. for all $C_{l}$
such that $\lim \inf_{\ell \rightarrow \infty }C_{l}/C_{\alpha
l}>0$ for all $\alpha >0$. Note that we are thus covering all
polynomial forms for $C_{l}^{-1}.$

Since (\ref{sufcon2}) only provides a sufficient condition for the
CLT, we need to analyze directly the more primitive condition
(\ref{as11}) for $m=0$ (however, the case $m\neq 0$ entails just a
more complicated notation). We
consider first an upper bound for the square root of the denominator of (\ref%
{as11}), which is given by $\widetilde{C}_{l}^{\left( 2\right) }$.

We have%
\begin{eqnarray*}
\widetilde{C}_{l}^{\left( 2\right) } &=&\sum_{j_{1},j_{2}}C_{j_{1}}C_{j_{2}}%
\frac{(2j_{1}+1)(2j_{2}+1)}{4\pi (2l+1)}\left( C_{j_{1}0j_{2}0}^{l0}\right)
^{2} \\
&\leq &2\sum_{j_{1},j_{2}}C_{j_{1}}C_{j_{2}}\frac{(2j_{1}+1)(2j_{2}+1)}{4\pi
(2l+1)}\left( C_{j_{1}0j_{2}0}^{l0}\right) ^{2} \\
&=&\frac{1}{2\pi }\sum_{j_{1}}C_{j_{1}}(2j_{1}+1)\sum_{j_{2}=j_{1}}^{\infty
}C_{j_{2}}\left( C_{j_{1}0l0}^{j_{2}0}\right) ^{2} \\
&\leq &\frac{1}{2\pi }\sum_{j_{1}}C_{j_{1}}(2j_{1}+1)\left\{ \sup_{j_{2}\geq
j_{1},j_{1}+j_{2}>l}C_{j_{2}}\right\} \sum_{j_{2}=0}^{\infty }\left(
C_{j_{1}0l0}^{j_{2}0}\right) ^{2}\leq KC_{l/2}\text{ .}
\end{eqnarray*}%
where we have used the relation $\frac{2j_{2}+1}{2l+1}%
(C_{j_{1}0j_{2}0}^{l0})^{2}=(C_{j_{1}0l0}^{j_{2}0})^{2}$, as well as%
\begin{equation*}
\sup_{j_{2}\geq j_{1},j_{1}+j_{2}>l}C_{j_{2}}\leq KC_{l/2}\text{ , and }%
\sum_{l=|l_{2}-l_{1}|}^{l_{2}+l_{1}}\left( C_{l_{1}0l_{2}0}^{l0}\right)
^{2}\equiv 1\text{ .}
\end{equation*}%
For the numerator of (\ref{as11}) one has that it is greater than%
\begin{eqnarray*}
&&\sum_{j_{1},j_{2}}\!C_{j_{1}}\!C_{j_{2}}\!\frac{(2j_{1}\!+\!1)(2j_{2}\!+\!1)}{%
(4\pi (2l\!+\!1))^{2}}\!\left\vert
\sum_{l_{1}}C_{l_{1}}(2l_{1}\!+%
\!1)C_{l_{1}0j_{1}0}^{l0}C_{l_{1}0j_{1}0}^{l0}C_{l_{1}0j_{2}0}^{l0}C_{l_{1}0j_{2}0}^{l0}\right\vert ^{2}
\\
&\geq &\sum_{j_{1},j_{2}}C_{j_{1}}C_{j_{2}}\frac{(2j_{1}+1)(2j_{2}+1)}{(4\pi
(2l+1))^{2}}\left\vert 5C_{2}\left\{
C_{20j_{1}0}^{l0}C_{20j_{2}0}^{l0}\right\} ^{2}\right\vert ^{2} \\
&\geq &C_{l}^{2}\frac{1}{(4\pi )^{2}}\left\vert 5C_{2}\left\{
C_{20l0}^{l0}\right\} ^{2}\right\vert ^{2}\geq KC_{l}^{2}.
\end{eqnarray*}%
The left-hand side of condition (\ref{as11}) is then bounded below by $%
\lim_{l\rightarrow \infty }$ $\left( K_{1}C_{l}^{2}\right)
/(K_{2}C_{l/2}^{2})\neq 0$, so that the CLT (\ref{as00}) cannot
hold.

\end{document}